\documentclass[10pt,reqno]{amsart}
\usepackage{latexsym,amsmath,amssymb,amscd}
\usepackage[all]{xy}
\usepackage{enumerate}

\def\today{\ifcase \month \or
   January \or February \or March \or April \or
   May \or June \or July \or August \or
   September \or October \or November \or December \fi
   \space\number\day , \number\year}
   
\makeatletter
  \newcommand\@dotsep{4.5}
  \def\@tocline#1#2#3#4#5#6#7{\relax
     \ifnum #1>\c@tocdepth 
     \else
     \par \addpenalty\@secpenalty\addvspace{#2}%
     \begingroup \hyphenpenalty\@M
     \@ifempty{#4}{%
     \@tempdima\csname r@tocindent\number#1\endcsname\relax
        }{%
         \@tempdima#4\relax
           }%
      \parindent\z@ \leftskip#3\relax \advance\leftskip\@tempdima\relax
      \rightskip\@pnumwidth plus1em \parfillskip-\@pnumwidth
       #5\leavevmode\hskip-\@tempdima #6\relax
       \leaders\hbox{$\m@th
       \mkern \@dotsep mu\hbox{.}\mkern \@dotsep mu$}\hfill
       \hbox to\@pnumwidth{\@tocpagenum{#7}}\par
       \nobreak
        \endgroup
         \fi}
\makeatother 

\begin{document}


\makeatletter
\@addtoreset{figure}{section}
\def\thefigure{\thesection.\@arabic\c@figure}
\def\fps@figure{h,t}
\@addtoreset{table}{bsection}

\def\thetable{\thesection.\@arabic\c@table}
\def\fps@table{h, t}
\@addtoreset{equation}{section}
\def\theequation{
\arabic{equation}}
\makeatother

\newcommand{\bfi}{\bfseries\itshape}

\newtheorem{theorem}{Theorem}
\newtheorem{acknowledgment}[theorem]{Acknowledgment}
\newtheorem{corollary}[theorem]{Corollary}
\newtheorem{definition}[theorem]{Definition}
\newtheorem{example}[theorem]{Example}
\newtheorem{lemma}[theorem]{Lemma}
\newtheorem{notation}[theorem]{Notation}
\newtheorem{problem}[theorem]{Problem}
\newtheorem{proposition}[theorem]{Proposition}
\newtheorem{question}[theorem]{Question}
\newtheorem{remark}[theorem]{Remark}
\newtheorem{setting}[theorem]{Setting}

\numberwithin{theorem}{section}
\numberwithin{equation}{section}

\renewcommand{\1}{{\bf 1}}
\newcommand{\Ad}{{\rm Ad}}
\newcommand{\Aut}{{\rm Aut}\,}
\newcommand{\ad}{{\rm ad}}
\newcommand{\botimes}{\bar{\otimes}}
\newcommand{\Ci}{{\mathcal C}^\infty}
\newcommand{\Cl}{{\rm Cl}\,}
\newcommand{\de}{{\rm d}}
\newcommand{\dr}{{\rm dr}\,}
\newcommand{\ee}{{\rm e}}
\newcommand{\End}{{\rm End}}
\newcommand{\id}{{\rm id}}
\newcommand{\ie}{{\rm i}}
\newcommand{\Index}{{\rm Index}}
\newcommand{\jump}{{\rm jump}}
\newcommand{\GL}{{\rm GL}}
\newcommand{\Gr}{{\rm Gr}}
\newcommand{\gen}{{\rm gen}}
\newcommand{\Hom}{{\rm Hom}\,}
\newcommand{\Ind}{{\rm Ind}}
\newcommand{\Int}{{\rm Int}\,}
\newcommand{\Ker}{{\rm Ker}\,}
\newcommand{\opp}{{\rm opp}}
\newcommand{\pr}{{\rm pr}}
\newcommand{\Ran}{{\rm Ran}\,}
\newcommand{\RRa}{{\rm RR}}
\newcommand{\rank}{{\rm rank}\,}
\renewcommand{\Re}{{\rm Re}\,}
\newcommand{\SO}{{\rm SO}\,}
\newcommand{\sa}{{\rm sa}}
\newcommand{\spa}{{\rm span}}
\newcommand{\supp}{{\rm supp}}
\newcommand{\tsr}{{\rm tsr}}
\newcommand{\Tr}{{\rm Tr}\,}

\newcommand{\CC}{{\mathbb C}}
\newcommand{\HH}{{\mathbb H}}
\newcommand{\PP}{{\mathbb P}}
\newcommand{\RR}{{\mathbb R}}
\newcommand{\TT}{{\mathbb T}}

\newcommand{\Ac}{{\mathcal A}}
\newcommand{\Bc}{{\mathcal B}}
\newcommand{\Cc}{{\mathcal C}}
\newcommand{\Ec}{{\mathcal E}}
\newcommand{\Fc}{{\mathcal F}}
\newcommand{\Hc}{{\mathcal H}}
\newcommand{\Ic}{{\mathcal I}}
\newcommand{\Jc}{{\mathcal J}}
\newcommand{\Kc}{{\mathcal K}}
\newcommand{\Lc}{{\mathcal L}}
\renewcommand{\Mc}{{\mathcal M}}
\newcommand{\Nc}{{\mathcal N}}
\newcommand{\Oc}{{\mathcal O}}
\newcommand{\Pc}{{\mathcal P}}
\newcommand{\Sc}{{\mathcal S}}
\newcommand{\Vc}{{\mathcal V}}
\newcommand{\Xc}{{\mathcal X}}
\newcommand{\Yc}{{\mathcal Y}}
\newcommand{\Zc}{{\mathcal Z}}
\newcommand{\Wc}{{\mathcal W}}

\newcommand{\Fg}{{\mathfrak F}}
\newcommand{\Ug}{{\mathfrak U}}
\newcommand{\Wg}{{\mathfrak W}}
\newcommand{\Zg}{{\mathfrak Z}}

\renewcommand{\gg}{{\mathfrak g}}
\newcommand{\hg}{{\mathfrak h}}
\newcommand{\kg}{{\mathfrak k}}
\newcommand{\mg}{{\mathfrak m}}
\newcommand{\sg}{{\mathfrak s}}
\newcommand{\rg}{{\mathfrak r}}

\newcommand{\zg}{{\mathfrak z}}

\newcommand{\KK}{\mathbb K}
\newcommand{\NN}{\mathbb N}
\newcommand{\QQ}{\mathbb Q}
\newcommand{\ZZ}{\mathbb Z}

\newcommand{\D}{U}

\makeatletter
\title[Linear dynamical systems of nilpotent Lie groups]{Linear dynamical systems of nilpotent Lie groups}
\author{Ingrid Belti\c t\u a and Daniel Belti\c t\u a}
\address{Institute of Mathematics ``Simion Stoilow'' 
of the Romanian Academy, 
P.O. Box 1-764, Bucharest, Romania}
\email{ingrid.beltita@gmail.com, Ingrid.Beltita@imar.ro}
\email{beltita@gmail.com, Daniel.Beltita@imar.ro}
\makeatother

\begin{abstract} 
We study the topology of orbits of dynamical systems defined by finite-dimensional representations of nilpotent Lie groups. 
Thus, the following dichotomy is established: either the interior of the set of regular points is dense in the representation space, or the complement of the set of regular points is dense, and then the interior of that complement is either empty or dense in the representation space. The regular points are by definition  the points whose orbits are locally compact in their relative topology. 
We thus generalize some results from the recent literature on linear actions of abelian Lie groups. 
As an application, we determine the generalized $ax+b$-groups whose $C^*$-algebras are antiliminary, that is, no closed 2-sided ideal is type~I.
\end{abstract}

\maketitle

\section{Introduction}

One of the main results of the present paper is that if $\pi\colon G\to\End(\Vc)$ is a finite-dimensional representation of a nilpotent Lie group and $\Gamma$ is the set of all regular points $v\in\Vc$ 
(i.e., the corresponding orbit $\pi(G)v$ is locally compact in its relative topology), then either the interior of $\Gamma$ is dense in $\Vc$ or the complement $\Vc\setminus \Gamma$ is a dense subset of $\Vc$ whose interior is either empty or is dense in~$\Vc$. 
(See Corollary~\ref{main_reg}.) 
When $G$ is an abelian Lie group, 
that dichotomy was essentially established in \cite{ArCuOu16} 
and, as shown in Remark~\ref{main_reg_ab} , it is closely related to 
a result of this type that had been obtained  in \cite[Ch. IV, Prop. 8.2]{Pu71} for the coadjoint representation of any connected, simply connected, solvable Lie group~$S$, and the latter result had been used for proving that almost all factors of the von Neumann algebra of $S$ are either type~I or type~II. 
In fact, one of the motivations of the present paper was to complete this picture by obtaining some information on the size of the set of unitary equivalence classes of the remaining, type~III, factor representations. 
This, if $G$ is abelian, we show that either $\Int \Gamma$ is dense in $\Vc$ or  the solvable Lie group $\Vc^*\rtimes G $ is antiliminary.
 (See Corollary~\ref{27August2019}, Remark~\ref{III} and Example~\ref{Mautner}.)

To obtain our results we use some techniques from linear algebra, namely rationality properties  of oblique projections and Moore-Penrose inverses, which allow us to improve the methods of investigation of the paper \cite{ArCuOu16} mentioned above. 
From this perspective, this paper is a sequel to our earlier study from \cite{BB17}. 

In Section~\ref{Sect_prel} we establish some auxiliary results of two types: reduction theory for regular orbits of locally compact group actions, and weight space decompositions with respect to nilpotent Lie groups. 
In Section~\ref{Sect_dichot} we develop the linear algebra tools we need for the investigation of regular points. 
In Section~\ref{Sect_reg} we establish the dichotomies for the sets of regular points in linear dynamical systems of general nilpotent Lie groups (Theorem~\ref{15March2019} and Corollary~\ref{main_reg}). 
Finally, in Section~\ref{Sect_ax+b} we give an application to representation theory of some solvable Lie groups that may not be type~I. 
Specifically, we give  
 a precise characterization of the so-called generalized $ax+b$-groups  whose $C^*$-algebra is antiliminary (Theorem~\ref{post-anti}).

 Besides representation theory \cite{ArCuDa19}, the study of regular points of dynamical systems holds an important role in several research areas, including for instance admissibility in continuous wavelet theory \cite{ArCuDaOu13}, \cite{BrCFM15}, \cite{ArCuOu16}, $C^*$-dynamical systems \cite{Wi07}, \cite{BB18}.
 
 \subsection*{General notation and terminology}
 We denote Lie groups by upper case Roman letters and their Lie algebras by the corresponding lower case Gothic letters. 
By a nilpotent Lie group we always understand a connected simply connected nilpotent Lie group.

For any subset of a real linear vector space $S\subset \Wc$, we denote by
$\langle S \rangle$ the additive subgroup of $(\Wc, +)$ generated by $S$. 

For any finite-dimensional real vector space $\Ug$,  the Grassmann manifold of $\Ug$ is the set $\Gr(\Ug)$ of all linear subspaces of $\Ug$.
When  $\Ug_0\subseteq\Ug$ is linear subspace,
we denote $\Gr_{\Ug_0}(\Ug):=\{\Wg\in\Gr(\Ug)\mid \Ug_0\dotplus\Wg=\Ug\}$. 
(See \cite{BB17} and \cite{ACM13}.)

 For any complex Hilbert space $\Hc$ we denote by $\Bc(\Hc)$ the von Neumann algebra of all bounded linear operators on $\Hc$, and by $\Kc(\Hc)$ the set of all compact operators on $\Hc$.

 For any real finite-dimensional Hilbert spaces $\Hc_1$,  $\Hc_2$ we use  again  the notation  $\Bc(\Hc_1, \Hc_2)$  for the space of the (bounded) $\RR$-linear operators $\Hc_1\to  \Hc_2$. When $\Vc$ is a $\KK$ vector space with $\KK=\RR$ or $\CC$, without considering any Hilbert space structure, 
 we denote by $\End_\KK(\Vc)$ the space of $\KK$-linear operators in $\Vc$.
When $\KK=\RR$ and no confusion arises, we write simply $\End(\Vc)$.

\section{Preliminaries}\label{Sect_prel}

In this section we collect two kinds of technical results for later
use in this paper.
Thus, we firstly develop a purely topological approach to the
regularity of group orbits, based on the idea of reduction of symmetry
groups that is used in other areas, for instance in geometric
mechanics. To this end, we actually establish on their natural level
of generality certain facts on Lie group actions that can be found in
\cite{ArCuDaOu13}. In the second part of this section we record some basic
properties of weight-space decompositions of representations of
nilpotent Lie algebras, a topic that is quite classical in Lie theory,
for instance in the study of Cartan subalgebras. Here we just point
out some features of the interaction between the weight-space
decompositions  and the operation of complexification. This topic was
also discussed in \cite{ArCuDaOu13} in the special case of abelian Lie
algebras.

\subsection{Reduction theory for regular orbits of locally compact group actions}

\begin{definition}
\normalfont 
For any fixed continuous action  $G\times X\to X$, $(g,x)\mapsto g.x$ of a locally compact group $G$ 
on a topological space $X$, a point $x_0\in X$ is called a \emph{regular point} if its corresponding orbit $G.x_0$
is locally compact, and then $G.x_0$ is called a \emph{regular orbit}. 
Here and everywhere in this paper, the orbits are endowed with their relative topology, that is, we regard the orbits as topological subspaces of the ambient space~$X$. 
\end{definition}

One of the reason why the regular orbits are interesting is the following result of \cite{aHWi02}, where
we use the notation $\Cc_0(X)$ for the continuous functions that vanish at $\infty$ on a locally comact space $X$.

\begin{lemma}\label{anHuefWilliams}
	Let $\alpha\colon G \times X\to X$ be a continuous action of a locally compact  group $H$ on a locally compact space $X$. 
	Assume that both $G$ and $X$ are second countable and $H$ is amenable.  
	If 
	 $H$ is abelian, 
	then the largest  
	type~I ideal of the crossed-product $C^*$-algebra
	$G\ltimes \Cc_0(X)$ is $G \ltimes \Cc_0(X_0)$, where $X_0$ is the set of all points $x\in X$ having a neighborhood $U_x$ for which the orbit of every point in $U_x$ is regular. 
\end{lemma}

\begin{proof}
	Since both $G$ and $X$ are second countable and $G$ is amenable, it follows that $G\ltimes \Cc_0(X)$ is EH-regular, as noted in \cite[page 537]{aHWi02}. 
	Then the assertion follows by \cite[Th. 3.16 and Rem. 3.17]{aHWi02}.
\end{proof}

\begin{remark}\label{Will}
	\normalfont
	In the case when $X$ is actually  a vector space $\Vc$ and $\alpha$ is an action by linear maps,  we can form the semidirect product group $\Vc \rtimes_\alpha G$. Then
	there is a $*$-isomorphism 
	$C^*(\Vc \rtimes_\alpha G)\simeq G \ltimes_{\alpha^*} \Cc_0(\Vc^*)$, by \cite[Ex. 3.16]{Wi07}, 
	where   $\alpha^{*}\colon G \times\Cc_0(\Vc^*)\to\Cc_0(\Vc^*)$, 
	$\alpha^{*}(g,f):=f\circ \alpha_{g^{-1}}$.
\end{remark}

The next lemma is a generalization of \cite[Cor. 2.3]{ArCuDaOu13} from Lie group actions to locally compact group actions.

\begin{lemma}\label{ACDO2.3_new}
	Let $G\times X\to X$, $(g,x)\mapsto g.x$, be a continuous action of a second countable, locally compact group on a Hausdorff topological space. 
	For any $x_0\in X$ with its orbit $\Oc:=G.x_0$, 
	then the following assertions are equivalent: 
	\begin{enumerate}[{\rm(i)}]
		\item\label{ACDO2.3_new_item1} The orbit $\Oc$ is regular. 
		 \item\label{ACDO2.3_new_item2} For every compact neighborhood $K$ of $\1\in G$ there exists a neighborhood $V$ of $x_0\in X$ with $V\cap \Oc\subseteq K.x_0$. 
		\item\label{ACDO2.3_new_item3} There exist a compact neighborhood $K$ of $\1\in G$ and a neighborhood $V$ of $x_0\in X$ with $V\cap \Oc\subseteq K.x_0$. 
	\end{enumerate}
If the space $X$ is first countable, the group $G$ is noncompact and countable at infinity, and the mapping $G\to X$, $g\mapsto g.x_0$, is injective, then the above assertions are equivalent to the following one: 
\begin{enumerate}[{\rm(i)}]
    \setcounter{enumi}{3}
	\item\label{ACDO2.3_new_item4} 
	There is no sequence $\{g_n\}_{n\in\NN}$ in $G$ satisfying $\lim\limits_{n\in\NN}g_n=\infty$ and $\lim\limits_{n\in\NN}g_n.x_0=x_0$. 
\end{enumerate}
If $X$ is locally compact and both $G$ and $X$ are noncompact, then \eqref{ACDO2.3_new_item5}$\Rightarrow$\eqref{ACDO2.3_new_item1}, 
where,  
\begin{enumerate}[{\rm(i)}]
	\setcounter{enumi}{4}
	\item\label{ACDO2.3_new_item5} 
	One has $\lim\limits_{g\to\infty}g.x_0=\infty$ in $X$. 
\end{enumerate}

\end{lemma}

\begin{proof} 
\eqref{ACDO2.3_new_item1}$\Rightarrow$\eqref{ACDO2.3_new_item2}
Let $L$ be any compact neighborhood of $\1\in G$ with $L^{-1}L\subseteq K$. 
One has $G=\bigcup\limits_{g\in G}gL$ hence, since $G$ is second countable, it has a subset $C\subseteq G$ that is at most countable and satisfies $G=\bigcup\limits_{c\in C}cL$. 
Since $\Oc=G.x_0$, we then obtain $\Oc=\bigcup\limits_{c\in C}cL.x_0$. 
This is a countable union of compact subsets of $\Oc$. 
Since $\Oc$ is locally compact by hypothesis, it then follows by Baire's theorem that there exists $c\in C$ for which the set $cL.x_0\subseteq \Oc$ has nonempty interior. 
Using the homeomorphism $\Oc\to\Oc$, $x\mapsto c^{-1}.x$, we then obtain that the interior of $L.x_0$ is nonempty, hence there exists $g\in L$ for which $g.x_0$ belongs to the interior of $L.x_0$. 
Then $x_0$ belongs to the interior of $g^{-1}L.x_0$. 
On the other hand $g^{-1}L.x_0\subseteq L^{-1}L.x_0\subseteq K.x_0$, 
hence $x_0$ belongs to the interior of $K.x_0$, which is equivalent to Assertion~\eqref{ACDO2.3_new_item2}.

\eqref{ACDO2.3_new_item2}$\Rightarrow$\eqref{ACDO2.3_new_item3} 
Obvious. 

\eqref{ACDO2.3_new_item3}$\Rightarrow$\eqref{ACDO2.3_new_item1} 
It follows by \eqref{ACDO2.3_new_item3} that $K.x_0$ is a compact neighborhood of $x_0\in \Oc$. 
Since the group $G$ acts transitively on $\Oc$ by homeomorphisms, 
it then follows that every point of $\Oc$ has a compact neighborhood, hence $\Oc$ is locally compact. 

Now assume that $G$ is noncompact and countable at infinity, hence it has a sequence of compact subsets $K_1\subseteq K_2\subseteq\cdots\subseteq G$ with $\bigcup\limits_{n\ge 1}K_n=G$. 
We also assume that $X$ is first countable and the mapping $G\to X$, $g\mapsto g.x_0$, is injective. 

\eqref{ACDO2.3_new_item3}$\Rightarrow$\eqref{ACDO2.3_new_item4} 
Let $\{g_n\}_{n\in\NN}$ be a sequence in $G$ with  $\lim\limits_{n\in\NN}g_n=\infty$ and $\lim\limits_{n\in\NN}g_n.x_0=x_0$. 
Using the notation of \eqref{ACDO2.3_new_item3}, there exists $n_V\in\NN$ with $g_n.x_0\in V\cap\Oc\subseteq K.x_0$ for every $n\ge n_V$. 
Since the mapping $G\to X$, $g\mapsto g.x_0$ is injective, we then obtain $g_n\in K$ for every $n\ge n_V$. 
Since the group $G$ is noncompact, this contradicts the assumption $\lim\limits_{n\in\NN}g_n=\infty$. 

\eqref{ACDO2.3_new_item4}$\Rightarrow$\eqref{ACDO2.3_new_item2} 
Reasoning by contradiction, we assume that there exists a compact neighborhood $K$ of $\1\in G$  such that for every neighborhood $V$ of $x_0\in X$ one has $V\cap \Oc\not\subseteq K.x_0$. 
Since $X$ is first countable, there exists a countable base of  neighborhoods $\{V_n\}_{n\in\NN}$ of $x_0\in X$. 
Then for every $n\in\NN$ one has by assumption $V_n\cap \Oc\not\subseteq K.x_0$, hence there exists $y_n\in V_n\cap\Oc$ with $y_n\not\in K.x_0$. 
Since $y_n\in\Oc$ and the mapping $G\to X$, $g\mapsto g.x_0$, is injective, there exists a unique element $g_n\in G$ with $y_n=g_n.x_0$. 
Moreover, since $y_n\in V_n$ for all $n\in\NN$, we obtain $\lim\limits_{n\in\NN}g_n.x_0=x_0$. 

It then follows by the hypothesis \eqref{ACDO2.3_new_item4} that one does not have $\lim\limits_{n\in\NN}g_n=\infty$. 
Then, selecting a suitable subsequence, one may assume that there exists $h\in G$ with $\lim\limits_{n\in\NN}g_n=h$, 
hence  $h.x_0=\lim\limits_{n\in\NN}g_n.x_0=x_0$
Since the mapping $G\to X$, $g\mapsto g.x_0$, is injective, it then follows that $h=\1\in G$, hence $\lim\limits_{n\in\NN}g_n=\1$. 
On the other hand $g_n.x_0=y_n\not\in K.x_0$ hence, by the injectivity of the mapping $g\mapsto g.x_0$ again, we obtain $g_n\not\in K$ for all $n\in\NN$. 
Since $K$ is a neighborhood of $\1\in G$, we thus obtain a contradiction. 

We now assume that the space $X$ is locally compact and noncompact. 

\eqref{ACDO2.3_new_item5}$\Rightarrow$\eqref{ACDO2.3_new_item1} 
Denote by $\widehat{G}=G\sqcup\{\infty_G\}$ and $\widehat{X}=X\sqcup\{\infty_X\}$ the one-point compactifications of $G$ and $X$, respectively, and define the continuous mapping $\psi\colon G\to X$, $\psi(g):=g.x_0$. 
It follows by the hypothesis~\eqref{ACDO2.3_new_item5} that  $\psi$ extends by continuity to a continuous mapping $\widehat{\psi}\colon\widehat{G}\to\widehat{X}$ with $\widehat{\psi}(\infty_G)=\infty_X$. 
Since $\widehat{G}$ is compact, it follows that $\widehat{\psi}(\widehat{G})$ is a compact subset of $\widehat{X}$. 
Then $\widehat{\psi}(\widehat{G})\setminus\{\infty_X\}$ is a locally closed subset of $\widehat{X}$. 
Since $\widehat{\psi}(\widehat{G})\setminus\{\infty_X\}\subseteq X$ and $X$ is an open subset of $\widehat{X}$, 
it also follows that $\widehat{\psi}(\widehat{G})\setminus\{\infty_X\}$ is a locally closed subset of $X$. 
Since $\widehat{\psi}(\widehat{G})\setminus\{\infty_X\}=\psi(G)=G.x_0=\Oc$, we thus see that $\Oc$ is a locally cosed subset of $X$, and we are done. 
\end{proof}

The following lemma is a generalization of \cite[Prop. 3.1]{ArCuDaOu13} from Lie group actions to locally compact group actions.
For a continous action $G \times X \to X$, $(g, x) \mapsto g.x$ of $G$ in a topological space $X$, and $x \in X$, we denote by $G(x)$ the corresponding isotropy group, that is,
$G(v):=\{g\in G\mid  g.x=x\}$.

\begin{lemma}\label{ACDO3.1_gen}
	Let $G\times X\to X$ and $G\times W\to W$ be continuous actions of a second countable, locally compact group 
	on Hausdorff topological spaces
	and assume that $q\colon X\to W$ is a continuous surjective mapping which is $G$-equivariant. 
	For any $x_0\in X$ we denote $w_0:=q(x_0)$, $\Oc:=G.x_0$, $\Oc_0:=G(w_0).x_0$, and $\Oc_{w_0}:=G.w_0$. 
	Then the following assertions hold: 
	\begin{enumerate}[{\rm(i)}]
		\item\label{ACDO3.1_gen_item1} One has $\Oc\cap q^{-1}(w_0)=\Oc_0$ and $q(\Oc)=\Oc_{w_0}$. 
		\item\label{ACDO3.1_gen_item2} If $\Oc$ is a locally 
		compact subset of $X$, then $\Oc_0$ is a locally 
		compact subset of~$X$. 
		\item\label{ACDO3.1_gen_item3} If $\Oc_0$ is a locally 
		compact subset of $X$ and $\Oc_{w_0}$ is a locally 
		compact subset of~$W$, then $\Oc$ is a locally 
		compact subset of $X$. 
		\item\label{ACDO3.1_gen_item4} If $\Oc_{w_0}$ is a locally  
		compact subset of~$W$, then one has 
		$$\Oc\text{ locally compact }\subseteq X\iff 
		\Oc_0\text{ locally compact }\subseteq X.$$
 	\end{enumerate}
\end{lemma}

\begin{proof}
	\eqref{ACDO3.1_gen_item1} 
	This is straightforward. 
	
	\eqref{ACDO3.1_gen_item2} 
	This follows by Assertion~\eqref{ACDO3.1_gen_item1} since the intersection of 
	a compact set with a closed set is compact. 
	
	\eqref{ACDO3.1_gen_item3}	
	Let $K$ be an arbitrary compact neighborhood of $\1\in G$, 
	and select any compact neighborhood 
	$K_1$ of $\1\in G$  
	with $K_1$ contained in the interior of $K$. 
	
	It follows by the hypothesis that $\Oc_{0}$ is locally compact along with 
	Lemma~\ref{ACDO2.3_new} and Assertion~\eqref{ACDO3.1_gen_item2} that  there exist 
	a neighborhood $V_1$ of $x_0\in X$ satisfying 
	$$V_1\cap \Oc\cap q^{-1}(w_0)\subseteq K_1.x_0.$$
	Using the continuity of the group action $G\times X\to X$ 
	at $(\1,x_0)\in G\times X$ we now select a compact neighborhood $K_0$ of $\1\in G$ and a neighborhood $V_2$ of $x_0\in X$ with $$K_0^{-1}.V_2\subseteq V_1.$$
	Since $K_1$ is a compact contained in the interior of $K$ we may shrink $K_0$ in order to have 
	$$K_0K_1\subseteq K.$$
	Using the hypothesis that $\Oc_{w_0}$ is locally compact in $W$, we obtain by Lemma~\ref{ACDO2.3_new} again a neighborhood $V_{w_0}$ of $w_0\in W$ with 
	$$V_{w_0}\cap \Oc_{w_0}\subseteq K_0.w_0.$$
	We now prove that $V:=V_2\cap q^{-1}(V_{w_0})$ is a neighborhood of $x\in X$ satisfying 
	$$V\cap\Oc\subseteq K.x_0$$
	and then $\Oc$ is locally closed by Lemma~\ref{ACDO2.3_new}. 
	
	It is easily seen that $V$ is a neighborhood of $x\in X$. 
	To prove the above inclusion, let $x\in V\cap\Oc$ arbitrary. 
	Then one has $q(x)\in q(q^{-1}(V_{w_0}))=V_{w_0}$ since $q$ is surjective, and on the other hand  $q(x)\in q(\Oc)=\Oc_{w_0}$ by Assertion~\eqref{ACDO3.1_gen_item1}, 
	hence $q(x)\in V_{w_0}\cap \Oc_{w_0}\subseteq K_0.w_0$. 
	Thus there exists $k_0\in K_0$ with $q(x)=k_0.w_0$, 
	and then $k_0^{-1}.x\in q^{-1}(w_0)$. 
	Since $x\in\Oc$, we also obtain $k_0^{-1}.x\in\Oc$. 
	Moreover, since $x\in V\subseteq V_2$, we obtain  
	$$k_0^{-1}.x\in q^{-1}(w_0)\cap \Oc\cap K_0^{-1}.V_2 
	\subseteq q^{-1}(w_0)\cap \Oc\cap V_1\subseteq K_1.x_0$$
	hence $x\in K_0K_1.x_0\subseteq K.x_0$. 
	
	\eqref{ACDO3.1_gen_item4} 
This directly follows by Assertions \eqref{ACDO3.1_gen_item2}--\eqref{ACDO3.1_gen_item3}. 
\end{proof}

\subsection{Weight space decompositions with respect to nilpotent Lie algebras}

\begin{notation}
\normalfont
We now introduce some notation to be used in the following, unless otherwise specified.
\begin{itemize}
	\item $\Vc$ is a finite-dimensional real vector space with $m:=\dim\Vc$;
	\item $\gg$ is a nilpotent Lie algebra with its corresponding Lie group $G=(\gg,\cdot)$, where $\cdot$ is given by the Baker-Campbell-Hausdorff formula; 
	\item $\gg_\CC:=\CC\otimes_\RR\gg=\gg\dotplus\ie\gg$ is the complexification of $\gg$, 
	with  its involutive antilinear map $\gg_\CC\to\gg_\CC$, $x\mapsto\bar x$ satisfying $\bar x=x$ for all $x\in\gg$; 
	\item the linear dual space of the complex Lie algebra $\gg_\CC$ is denoted by $\gg_\CC^*$, 
	with its involutive antilinear map $\gg_\CC^*\to\gg_\CC^*$, $\lambda\mapsto\bar\lambda$, where $\bar\lambda(x):=\overline{\lambda(\bar x)}$ for all $x\in\gg_\CC$ and $\lambda\in\gg_\CC^*$;
	\item $\rho\colon\gg\to\End(\Vc)$ is a Lie algebra representation, 
	with its corresponding extension to a morphism of complex Lie algebras $\rho\colon\gg_\CC\to\End_\CC(\Vc_\CC)$; 
	\item If  $\Vc_\CC:=\CC\otimes_\RR\Vc$ denotes the complexification of $\Vc$, then $C\colon\Vc_\CC\to\Vc_\CC$ is the involutive antilinear map satisfying $Cv=v$ for all $v\in\Vc$.
\end{itemize}
\end{notation}

\begin{remark}\label{extra1}
\normalfont
For every $x\in \gg_\CC$, $C\rho(x)C=\rho(\bar x)$. 
Indeed, let $x_1,x_2\in\gg$ with $x=x_1+\ie x_2$. 
Then $C\rho(x)C=C(\rho(x_1)+\ie\rho(x_2))C=\rho(x_1)-\ie\rho(x_2)=\rho(\bar x)$.
\end{remark}

\begin{lemma}\label{act1}
For every $\lambda\in\gg_\CC^*$, the set 
$$\Vc_\CC^\lambda:=
\bigcap\limits_{x\in\gg_\CC}\Ker(\rho(x)-\lambda(x)\id)^m\subseteq\Vc_\CC$$
is an invariant subspace for the representation $\rho\colon\gg_\CC\to\End_\CC(\Vc_\CC)$, 
and one has 
$$\Vc_\CC=\bigoplus\limits_{\lambda\in\Lambda}\Vc_\CC^\lambda$$
where the set $\Lambda:=\{\lambda\in\gg_\CC^*\mid \Vc_\CC^\lambda\ne\{0\} \}$ is finite and for every $\lambda\in\Lambda$ one has $[\gg_\CC,\gg_\CC]\subseteq\Ker\lambda$.  
\end{lemma}

\begin{proof}
Since $\gg_\CC$ is a complex nilpotent Lie algebra, the assertions follow for instance by \cite[Th. 2.9]{Ca05}.
\end{proof}

\begin{lemma}\label{act2}
For all $\lambda\in\gg_\CC^*$ one has $C(\Vc_\CC^\lambda)=\Vc_\CC^{\bar\lambda}$.
\end{lemma}

\begin{proof}
For every $v\in\Vc_\CC^\lambda$ and $x\in\gg_\CC$ one has 
$(\rho(x)-\lambda(x)\id)^mv=0$ hence, since  $C^2=\id$, 
$$0=C(\rho(x)-\lambda(x)\id)^mv=(C(\rho(x)-\lambda(x)\id)C)^mCv
=(C\rho(x)C-\overline{\lambda(x)}\id)^mCv$$
By Remark~\ref{extra1}, $C\rho(x)C=\rho(\bar x)$,  
and thus the above equality implies
$$(\forall x\in\gg_\CC)\quad 
(\rho(\bar x)-\overline{\lambda(x)}\id)^mCv=0$$
hence 
$$(\forall x\in\gg_\CC)\quad 
(\rho(x)-\overline{\lambda(\bar x)}\id)^mCv=0$$
Therefore $Cv\in \Vc_\CC^{\bar\lambda}$. 
Thus $C(\Vc_\CC^\lambda)\subseteq \Vc_\CC^{\bar\lambda}$, 
and the converse inclusion follows by symmetry, since $C^2=\id$. 
This completes the proof. 
\end{proof}

\begin{definition}\label{act3}
\normalfont
For any subset $\Phi\subseteq\Lambda$ we denote 
$\Vc_\CC^\Phi:=\bigoplus\limits_{\lambda\in\Phi}\Vc_\CC^\lambda$ 
and we define the idempotent mapping $P_\Phi\in\End_\CC(\Vc)$ with 
$\Ran P_\Phi=\Vc_\CC^\Phi$ and $\Ker P_\Phi=\Vc_\CC^{\Lambda\setminus\Phi}$, using Lemma~\ref{act1}. 
If one has $\Phi=\{\lambda\}$ for some $\lambda\in\Lambda$, then we define $P_\lambda:=P_\Phi=P_{\{\lambda\}}$. 

We denote $\Lambda':=\{\lambda\in\Lambda\mid\bar{\lambda}=\lambda\}$ 
and, using Lemma~\ref{act2}, we fix a subset $\Lambda''\subseteq\Lambda\setminus\Lambda'$ satisfying $\Lambda\setminus\Lambda'=\Lambda''\sqcup\overline{\Lambda''}$, 
where $\overline{\Lambda''}:=\{\bar{\lambda}\mid\lambda\in\Lambda''\}$.
We then obtain the partition $\Lambda=\Lambda'\sqcup\Lambda''\sqcup\overline{\Lambda''}$.  
This leads to the direct sum decomposition  
\begin{equation}\label{act3_eq1}
\Vc_\CC=
\Vc_\CC^{\Lambda'}\dotplus\Vc_\CC^{\Lambda''}
\dotplus\Vc_\CC^{\overline{\Lambda''}}
\end{equation}
via Lemma~\ref{act1} again. 
\end{definition}

\begin{lemma}\label{act4}
The mapping 
$$\iota\colon \Vc\to(\Vc\cap \Vc_\CC^{\Lambda'})\dotplus\Vc_\CC^{\Lambda''},\quad 
\iota(v):=P_{\Lambda'}v+P_{\Lambda''}v$$
is an $\RR$-linear isomorphism. 
\end{lemma}

\begin{proof}
One can easily check that  
$\dim_\RR\Vc=\dim_\RR((\Vc\cap \Vc_\CC^{\Lambda'})\dotplus\Vc_\CC^{\Lambda''})$ and $\Ker\iota=\{0\}$, 
using Lemmas \ref{act2} and \ref{act1}. 
\end{proof}

\begin{definition}
\normalfont
We define the real vector space 
$$\widetilde{\Vc}:=(\Vc\cap \Vc_\CC^{\Lambda'})\dotplus\Vc_\CC^{\Lambda''}$$
and the Lie algebra representation
$$\widetilde{\rho}\colon\gg\to\End(\widetilde{\Vc}),\quad 
\widetilde{\rho}(x):=\iota\circ\rho(x)\circ\iota^{-1},$$
using Lemma~\ref{act4}. 
\end{definition}

\section{Dichotomies for rational functions and oblique projections}\label{Sect_dichot}
In this section we obtain some auxiliary results that are needed in Section~\ref{Sect_reg}.
\subsection{Dichotomies for rational functions}
The aim of this subsection is to obtain a version (Lemma~\ref{dic3})  of  some well-known results (see Remark~\ref{dic_finite}), concerning finite families of rational $\Yc$-valued functions rather than vectors of the finite-dimensional real vector space~$\Yc$. 

\begin{remark}\label{dic_finite}
	\normalfont
	Let $F\subseteq \Yc$ be a finite subset of a finite-dimensional real vector space. 
	Then the countable subgroup $\langle F\rangle=\sum\limits_{f\in F}\ZZ f$ of $(\Yc,+)$ generated by $F$ is locally closed if and only if $\langle F\rangle$ is closed, and  if and only if $\langle F\rangle$ is discrete. 
	
	In fact, every locally closed subgroup of a topological group is closed. 
	Furthermore, if $\langle F\rangle$ is closed in $\Yc$ and $\langle F\rangle$ is not discrete, then by \cite[Ch. 5, \S 1, no. 1, Cor.]{Bo74} there exists $y\in\Yc$ with $\RR y\subseteq\langle F\rangle$. 
	This implies that $\langle F\rangle$ is not countable, which is a contradiction. 
	\end{remark}

\begin{lemma}\label{dic0}
Let $\Yc$ be a finite-dimensional real vector space and $y_1,\dots,y_k\in\Yc$ be linearly independent vectors. 
For any positive integer $r\ge 1$ and any real numbers $a_{ij}$ for $1\le i\le r$ and $1\le j\le k$ define 
$z_i:=\sum\limits_{j=1}^ka_{ij}y_j\in\Yc$. 
Then the subgroup $\ZZ y_1+\cdots+\ZZ y_k+\ZZ z_1+\cdots+\ZZ z_r$ of $(\Yc,+)$ 
generated by $y_1,\dots,y_k,z_1,\dots,z_r$ is discrete/closed/locally closed 
if and only if for all $1\le i\le r$ and $1\le j\le k$ one has $a_{ij}\in\QQ$. 
\end{lemma}

\begin{proof}
Using Remark~\ref{dic_finite}, the case $r=1$ follows by \cite[Ch. 5, \S 1, no. 1, Cor.]{Bo74}. 
The general case can then be proved by iteration, using the fact that any subgroup of a discrete group is discrete. 
\end{proof}

\begin{definition}\label{dic_def}
\normalfont 
For any finite-dimensional real vector space $\Xc$ and any 
subset $\D\subseteq\Xc$ we say that a function $f\colon \D\to\RR$ is \emph{rational} if there exist polynomial functions $P,Q\colon\Xc\to\RR$ satisfying $Q(x)\ne0$ and $f(x)=P(x)/Q(x)$ for all $x\in \D$. 
If $\Yc$ is another finite-dimensional real vector space, then a vector function $\phi\colon \D\to\Yc$ is \emph{rational} if for every linear functional $\xi\colon\Yc\to\RR$ the scalar function $\xi\circ\varphi\colon \D\to\RR$ is rational in the above sense.
\end{definition}

\begin{lemma}\label{dic1}
Let $\Xc$ and $\Yc$ be finite-dimensional real vector spaces, and fix any countable subset $S\subseteq\Yc$. 
If $\D\subseteq \Xc$ is 
an open subset and $f\colon \D\to\Yc$ is a 
 rational function, then one of the following assertions holds: 
\begin{enumerate}[{\rm(i)}]
	\item\label{1} There exists $s\in S$ with $f(\D)=\{s\}$. 
	\item\label{2} The subset $f^{-1}(\Yc\setminus S)\subseteq \D$ is dense and, if moreover $\Yc\setminus S$ is totally disconnected, then either $\Int f^{-1}(\Yc\setminus S)=\emptyset$ or there exists $y\in\Yc\setminus S$ with $f(\D)=\{y\}$. 
\end{enumerate}
\end{lemma} 

\begin{proof}
	The assertions \eqref{1} and \eqref{2} cannot hold true simultaneously. 
	Therefore it suffices to prove that if \eqref{1} fails, then \eqref{2} is true. 
	
For every $s\in S$ the set $f^{-1}(s)$ is closed and, since $f$ is 
rational, we have that 
either $\Int f^{-1}(s)=\emptyset$ or $f^{-1}(s)=\D$. 
If the first assertion in the statement fails to be true, it then follows that for every $s\in S$ the subset $f^{-1}(s)\subseteq \D$ is closed and has empty interior. 
Then, by Baire's category theorem the set $\bigcup\limits_{s\in S}f^{-1}(s)$ has empty interior in~$\D$, that is, its complement is dense in~$\D$. 
One has $\D\setminus\bigcup\limits_{s\in S}f^{-1}(s)=f^{-1}(\Yc\setminus S)$, hence $f^{-1}(\Yc\setminus S)$ is a dense subset of~$\D$. 

Now assume that moreover $\Yc\setminus S$ is totally disconnected and $\Int f^{-1}(\Yc\setminus S)\ne\emptyset$. 
Then there exists an open connected  nonempty subset $B\subseteq \D$ with $f(B)\subseteq \Yc\setminus S$. 
Since $B$ is connected and the function $f$ is continuous, 
the set $f(B)$ is in turn connected. 
The hypothesis that $\Yc\setminus S$ is totally disconnected then implies that there exists $y_0\in\Yc\setminus S$ with $f(B)=\{y_0\}$. 
Since the set $B$ is open and nonempty, and the function $f$ is rational, it then follows that $f$ is constant, hence $f(\D)=\{y_0\}$. 
This completes the proof. 
\end{proof}

\begin{lemma}\label{dic2}
	Let $\Xc$ be a finite-dimensional real vector space and  $\D\subseteq\Xc$ be 
	an open 
	subset. 
	Let $1\le k\le m$ be any positive integers and assume that  $y_j\colon \D\to\RR^m$ and $z_i\colon \D\to\RR^m$ are  
	rational functions for $j=1,\dots,k$ and $i=1,\dots,r$, satisfying the conditions 
	$$y_1(x_0),\dots,y_k(x_0)\in\RR^m\text{ are linearly independent at some point }x_0\in \D$$
	and 
	$$z_1(x),\dots,z_r(x)\in\spa_{\RR}\{y_1(x),\dots,y_k(x)\}\text{ for all }x\in \D.$$ 
	If we denote 
	$$A:=\{x\in \D\mid z_1(x),\dots,z_r(x)\in\spa_{\QQ}\{y_1(x),\dots,y_k(x)\}\}$$
	then there exists a dense open subset $\D_0\subseteq \D$ for which 
	exactly one of the following assertions holds: 
	\begin{enumerate}[{\rm(i)}]
		\item 
		One has $\D_0\subseteq A$ and there exist 
		$a_{ij}\in\QQ$ for $1\le i\le r$ and $1\le j\le k$ with 
		$z_i(x)=\sum\limits_{j=1}^ka_{ij}y_j(x)$ for $i=1,\dots,r$ and all $x\in \D_0$. 
		\item The set $\D\setminus A$ is dense in $\D$, and either $\Int(\D\setminus A)=\emptyset$ or there exist 
		$a_{ij}\in\RR$ for $1\le i\le r$ and $1\le j\le k$ with 
		$z_i(x)=\sum\limits_{j=1}^ka_{ij}y_j(x)$ for $i=1,\dots,r$ and all $x\in \D_0$. 
	\end{enumerate}
\end{lemma}

\begin{proof}
	Assume $\D\ne\emptyset$ and write 
	$$y_j=(y_{j1},\dots,y_{jm})\colon \D\to\RR^m 
	\text{ and }z_i=(z_{i1},\dots,z_{im})\colon \D\to\RR^m$$ 
	for $j=1,\dots, k$ and $i=1,\dots,r$. 
	Since the vectors $y_1(x_0),\dots,y_k(x_0)\in\RR^m$ are linearly independent at some point $x_0\in \D$, there exists a set  $J=\{t_1,\dots,t_k\}\subseteq\{1,\dots,m\}$ with $1\le t_1<\cdots<t_k\le m$ and $x_0\in \D_J$, where
	$\D_J:=\{x\in \D\mid \delta_J(x)\ne0\}$ and 
	$$\delta_J\colon \D\to\RR,\quad \delta_J(x):=\det(y_{s,t_j}(x))_{1\le s,j\le k}.$$ 
	It is clear that $\delta_J\colon \D\to\RR$ is a 
	rational 
	 function hence 
	 the set $\D_J$ is a dense open subset of~$\D$. 
	
	Moreover, 
	the vectors $y_1(x),\dots,y_k(x)\in\RR^m$ are linearly independent for every $x\in \D_J$. 
	It then follows by the hypothesis that for $i=1,\dots,r$ and every $x\in \D_J$ there exist uniquely determined $a_{i1}(x),\dots,a_{ik}(x)\in\RR$ with $z_i(x)=\sum\limits_{j=1}^ka_{ij}(x)y_j(x)$, that is, 
	$$\begin{cases}
	z_{i1}(x)=a_{i1}(x)y_{11}(x)+a_{i2}(x)y_{21}(x)+\cdots+a_{ik}(x)y_{k1}(x), \\
	\makebox[4em]{\dotfill} \\
	z_{im}(x)=a_{i1}(x)y_{1m}(x)+a_{i2}(x)y_{2m}(x)+\cdots+a_{ik}(x)y_{km}(x).
	\end{cases}
	$$
	If $B_J(x):=(y_{s,t_j}(x))_{1\le s,j\le k}\in M_k(\RR)$, 
	then $\det B_J(x)=\delta_J(x)\ne 0$ for all $x\in \D_J$, 
	hence the above overdetermined linear system  can be solved for $a_{i1}(x),\dots,a_{ik}(x)$ by 
	$$\underbrace{(z_{i,t_1}(x)\  \dots\  z_{i,t_k}(x))}_{\in M_{1,k}(\RR)} \underbrace{B_J(x)^{-1}}_{\in M_k(\RR)} 
	=\underbrace{(a_{i1}(x) \ \dots\  a_{ik}(x))}_{\in M_{1,k}(\RR)}.$$
	This shows that $a_{i1},\dots,a_{ik}$ are 
	rational 
	 functions on $\D_J$ for $i=1,\dots,r$.

	Then the function $f:=(a_{ij}\vert_{\D_J})_{1\le 1\le r,1\le j\le k}\colon \D_J\to M_{rk}(\RR)$ is 
	 rational and, for $S:=M_{rk}(\QQ)\subseteq M_{rk}(\RR)$, one has $A\cap \D_J=f^{-1}(S)$ and moreover $M_{rk}(\RR)\setminus S$ is totally disconnected.  
	Thus, by Lemma~\ref{dic1}, 
	exactly one of the following cases can occur: 
	\begin{itemize}
	\item There exists $s\in S$ with $f(x)=s$ for all $x\in \D_J$ 
	(in particular $\D_J\subseteq A$).  
	\item The set $f^{-1}(M_{rk}(\RR)\setminus S)$ is dense in $\D_J$, and either $\Int f^{-1}(M_{rk}(\RR)\setminus S)=\emptyset$ or there exists $F_0\in M_{rk}(\RR)\setminus S$ with $f(\D_J)=\{F_0\}$. 
    \end{itemize}
	Since $\D_J$ is a dense subset of $\D$ and $f^{-1}(M_{rk}(\RR)\setminus S)= \D_J\setminus A$, 
	one may set $\D_0:=\D_J$ to complete the proof. 
\end{proof}

\begin{lemma}\label{dic3}
	Let $\Xc$ and $\Yc$ a finite-dimensional real vector spaces and  $\D\subseteq\Xc$ be 
	an open 
	subset. 
Let $f_1,\dots,f_q\colon \D\to\Yc$ be rational functions and define 
$$A:=\{x\in \D\mid \langle f_1(x),\dots,f_q(x)\rangle\text{ is closed in }\Yc\}.$$
Then exactly one of the following assertions holds: 
	\begin{enumerate}[{\rm(i)}]
		\item The set $\Int A$ is a dense open subset of $\D$. 
		\item The set $\D\setminus A$ is dense in $\D$, and either $\Int(\D\setminus A)=\emptyset$ or $\Int(\D\setminus A)$ is dense in $\D$.
	\end{enumerate}
\end{lemma}

\begin{proof}
For every $x\in \D$ we define $\Ec_x:=\spa_{\RR}\{f_j(x)\mid 1\le j\le q\}$, 
and we fix a point $x_0\in \D$ with $\dim \Ec_x\le \dim\Ec_{x_0}$ for all $x\in \D$. 
We denote $k:=\dim\Ec_{x_0}$, $r:=q-k$, and we label $f_1,\dots,f_q$ as $y_1,\dots,y_k,z_1,\dots,z_r$, 
with $y_1(x_0),\dots,y_k(x_0)$ being a basis of $\Ec_{x_0}$. 

Since $y_1,\dots,y_k\colon \D\to \Yc$ are rational functions and  $y_1(x_0),\dots,y_k(x_0)$ are linearly independent, 
it is straightforward to obtain a Zariski open subset $\D_0\subseteq \D$ with $x_0\in \D_0$, 
for which $y_1(x),\dots,y_k(x)$ are linearly independent for all $x\in \D_0$, as in the proof of Lemma~\ref{dic2}. 
Since $\dim\Ec_x\le k$ for all $x\in \D$, it then follows that $y_1(x),\dots,y_k(x)$ is a basis of $\Ec_x$ for all $x\in \D_0$. 
In particular, for every $x\in \D_0$ one has 
\begin{equation}\label{dic3_proof_eq1}
z_1(x),\dots,z_r(x)\in\Ec_x=\spa_{\RR}\{y_1(x),\dots,y_k(x)\}\text{ for all }x\in \D_0.
\end{equation} 
On the other hand, denoting 
$$A_0:=\{x\in \D_0\mid z_1(x),\dots,z_r(x)\in\spa_{\QQ}\{y_1(x),\dots,y_k(x)\}\}$$
we obtain $A_0=A\cap \D_0$ by Lemma~\ref{dic0}. 
Moreover, it follows by \eqref{dic3_proof_eq1} along with Lemma~\ref{dic2} that there exists a dense open subset $\D_{00}\subseteq \D_0$ for which exactly one of the following cases occurs: 
\begin{itemize}
\item One has $\D_{00}\subseteq A_0$. 
\item The set $\D_0\setminus A_0$ is dense in $\D_0$, and either $\Int(\D_0\setminus A_0)=\emptyset$ or $\Int(\D_0\setminus A_0)$ is dense in $\D_0$. 
\end{itemize}
Since $\D_0$ is a dense open subset of $\D$, this concludes the proof. 
\end{proof}

\subsection{Oblique projections and an application of Moore-Penrose inverses}

We start by recalling the definition of the Moore-Penrose inverse of a bounded operator in a finite dimensional Hilbert space. (See \cite[subsect. 5.5.4]{GvL96}) for the next definition, and  \cite{BB17} for the relevance of the Moore-Penrose inverse in the context of oblique projections and more  references.)

\begin{definition}
	\normalfont
	Let $\Hc$ be any finite-dimensional real Hilbert space. 
	For every operator $A\in\Bc(\Hc)$ its \emph{Moore-Penrose inverse} is the operator 
	$A^\dagger:=B\in\Bc(\Hc)$ that is uniquely determined by the equations 
	$$ABA=A,\ BAB=B,\ (AB)^*=AB,\ (BA)^*=BA$$
		Then $A^\dagger$ exists for every $A\in\Bc(\Hc)$.
\end{definition}

\begin{definition}
\normalfont 
For any fixed integer $n\ge 1$ we define the polynomial functions $\gamma_0,\gamma_1,\dots,\gamma_m\colon \Bc(\RR^n,\RR^m)\to\RR$ by the equation  
$$(\forall t\in\RR)(\forall A\in \Bc(\RR^n,\RR^m))\quad 
\det(\1+tAA^\top)=\sum_{k=0}^m\gamma_k(A)t^k.$$
It is easily seen that $\gamma_0(A)=1$ and $\gamma_m(A)=\det(AA^\top)$ for all $A\in \Bc(\RR^n,\RR^m)$. 
The functions $\gamma_0,\gamma_1,\dots,\gamma_m$ are called the \emph{Gram coefficients}; cf. \cite[Def. 1.2]{DGL05}.
\end{definition}

\begin{lemma}\label{MP}
For any integers $m,n\ge 1$ and $A\in \Bc(\RR^n,\RR^m)$ the following assertions hold: 
\begin{enumerate}[{\rm(i)}]
	\item If $r\in\{1,\dots,n\}$, one has $\dim(\Ker A)=n-r$ 
	if and only if $\gamma_r(A)\ne0$ and 
	$$(\forall t\in\RR)\quad 
	\det(\1+tAA^\top)=\sum_{k=0}^r\gamma_k(A)t^k.$$
	\item If $\dim(\Ker A)=n-r$, then the orthogonal projection onto $\Ker A$ is given by 
	$$P_{\Ker A}=\1-\frac{1}{\gamma_r(A)}\sum_{k=0}^{r-1}
	(-1)^{r-1-k}\gamma_k(A)(A^\top A)^{r-k}.$$
	\item If $\dim(\Ker A)=n-r$, then the Moore-Penrose inverse of $A$ is given by 
	$$A^\dagger=\frac{1}{\gamma_r(A)}\Bigl(\sum_{k=0}^{r-1}
	(-1)^{r-1-k}\gamma_k(A)(A^\top A)^{r-1-k})\Bigr)A^\top\in\Bc(\RR^m,\RR^n).$$
\end{enumerate}
\end{lemma}

\begin{proof}
See \cite[Lemmas 1.3--1.4 and Prop. 1.6]{DGL05}.  	
\end{proof}

We recall the following definition.

\begin{definition}\label{oblpr}
	\normalfont
If $\Ug$ is a finite-dimensional real vector space with two subspaces $\Ug_1,\Ug_2\subseteq\Ug$ with $\Ug=\Ug_1\dotplus\Ug_2$, 
then the corresponding linear \emph{oblique projection of $\Ug$ onto $\Ug_2$ along $\Ug_1$} 
is the linear operator $E_{\Ug_2}\colon \Ug\to\Ug$ defined by the conditions $\Ker E_{\Ug_2}=\Ug_1$ and $E_{\Ug_2}w=w$ for every $w\in\Ug_2$. 
\end{definition}

\begin{proposition}\label{norm0}
Let $\Ug$, $\Xc$, and $\Yc$ be finite-dimensional real vector spaces. 
If $\D\subseteq\Xc$, $\Ug_0\in\Gr(\Ug )$, and 
$T\colon \D\to \Bc(\Ug,\Yc)$ is a rational mapping with $\Ker T(x)\in\Gr_{\Ug_0}(\Ug)$ for all $x\in \D$,  
then the mapping $\theta\colon \D\to \Bc(\Ug,\Ug_0)$, $\theta(x):=E_{\Ug_0}(\Ker T(x))$ is rational. 
\end{proposition}

\begin{proof}
Fix a scalar product on $\Ug$, so that $\Ug$ becomes a real Hilbert space. 
It follows by 
 \cite[Lemma 2.3]{BB17} 
that 
$$(\forall x\in \D)\quad 
\theta(x)=P_{\Ker T(x)}((\1-P_{\Ug_0})P_{\Ker T(x)})^\dagger (\1-P_{\Ug_0}).$$
On the other hand, denoting $r:=\dim\Ug_0$, for all $x\in \D$ one has by hypothesis $\Ker T(x)\in\Gr_{\Ug_0}(\Ug)$ 
hence $\dim(\Ker T(x))=\dim\Ug-r$. 
It is also easily checked that 
$$(\forall x\in \D)\quad 
\Ker((\1-P_{\Ker T(x)})P_{\Ug_0})=\Ug_0^\perp.$$ 
It then follows by Lemma~\ref{MP} along with the above fomula that $\theta$ is a composition of rational mappings, 
hence $\theta$ is in turn rational a rational function. 
\end{proof}

The following example is needed in the proof of Corollary~\ref{main_reg}.
\begin{example}[Stratification according to a unipotent representation]
	\label{strat}
\normalfont
Let $\gg$ be a nilpotent Lie algebra with $m:=\dim\gg$ and its corresponding group $G=(\gg, \cdot)$, $\Vc$~be a finite-dimensional real vector space with $n:=\dim\Vc$, and $\rho\colon G\to\End(\Vc)$ be a unipotent representation, that is, $(\1-\rho(x))^n=0$ for all $x\in G$.  
The contragredient representation $\rho^*\colon G\to\End(\Vc^*)$, $\rho^*(x):=\rho(-x)^*$, is unipotent as well. 
For a Jordan-H\"older sequence of ideals of $\gg$, 
$$\Fg_\bullet:\quad\{0\}=\Fg_0\subseteq\Fg_1\subseteq\cdots\subseteq\Fg_m=\gg$$  
and a sequence of linear subspaces 
$$\Yc_\bullet:\quad \{0\}=\Yc_0\subseteq\Yc_1\subseteq\cdots\subseteq\Yc_n=\Vc^*$$ 
with $\dim\Yc_j=j$ and $\de\rho^*(\gg)\Yc_j\subseteq\Yc_{j-1}$ for $j=1,\dots,n$,  
and a linear subspace $\hg\subseteq\gg$
we denote 
 $$\jump_{\Fg_\bullet}(\hg):=\{j\in\{1,\dots,m\}\mid \Fg_{j-1}+\hg\subsetneqq\Fg_j+\hg\}.$$ 
 By the construction of \cite[\S 1.4]{Pe94} we obtain  a stratification of $(\Vc^*)^*=\Vc$ according to the representation~$\rho^*$. 
That is, we obtain a partition 
$$\Vc=\bigsqcup\limits_{\varepsilon\in{\boldsymbol\Ec}}
\Omega_\varepsilon, $$
where the set ${\boldsymbol\Ec}$ is finite and is endowed with a total ordering $\prec$ satisfying the following conditions: 
\begin{enumerate}[{\rm(i)}]
	\item\label{strat_item1} For every $\varepsilon\in{\boldsymbol\Ec}$ the set $\Omega_\varepsilon$ is $\rho(G)$-invariant and there exists a subset $e(\varepsilon)\subseteq\{1,\dots,m\}$ 
	with $\jump_{\Fg_\bullet}(\gg(v))=e(\varepsilon)$ 
	 for all $v\in\Omega_\varepsilon$, where $\gg(v)$ is the isotropy subalgebra for $\rho$ at $v$, that is, the Lie algebra of $G(v)$. 
	\item\label{strat_item2} There exists a family of polynomial functions $\{Q_\varepsilon\colon\Vc\to\RR\mid \varepsilon\in{\boldsymbol\Ec}\}$ with 
	$\Omega_\varepsilon=\{v\in\Vc\mid Q_\varepsilon(v)\ne0=Q_{\varepsilon'}\text{ if }\varepsilon'\prec\varepsilon\}$. 
	(See \cite[Prop.~1.3.2]{Pe94}.)
	In particular, if $\varepsilon_0=\min{\boldsymbol\Ec}$, then $\Omega_{\varepsilon_0}$ is a dense open subset of~$\Vc$. 
\end{enumerate}
In this setting we also recall for later use that for any selection $x_j\in\gg_j\setminus\gg_{j-1}$ for $j=1,\dots,m$,  
\begin{equation}\label{strat_eq1}
\text{if $e:=\jump_{\Fg_\bullet}(\hg)$ and 
	$\gg_{e,\hg}:=\spa\{x_j\mid j\in e\}$, then 
	$\hg\in\Gr_{\gg_{e,\hg}}(\gg)$,}
\end{equation}
that is, there is the direct sum decomposition 
$\gg=\hg\dotplus\gg_{e,\hg}$. 
(See \cite[Prop. 3.4(x)]{BB17}.) 
We  note  that actually
\begin{equation}\label{strat_eq2}
\Fg_j=(\hg\cap\Fg_j)\dotplus(\gg_{e,\hg}\cap\Fg_j)\text{ for }j=1,\dots,m.  
\end{equation} 
To see this, let $q:=E_{\hg,\gg}(\gg_{e,\hg})\colon \gg\to\gg$ be the oblique projection onto $\hg$ along $\gg_{e,\hg}$ given by the direct sum decomposition $\gg=\hg\dotplus\gg_{e,\hg}$. 
It suffices to show that $q(\Fg_j)\subseteq\Fg_j$ for $j=1,\dots,n$.  
If $j=1$, then one has either $1\in e$ and then $q(x_1)=0\in\Fg_1$, 
or $1\not\in e=\jump_{\Fg_\bullet}(\hg)$ and then $\Fg_1\subseteq\hg$ hence $q(x_1)=x_1\in\Fg_1$. 
In any of these cases we obtain $q(\Fg_1)=\RR q(x_1)\subseteq\Fg_1$. 
Now let $2\le j\le m$ and assume that $q(\Fg_{j-1})\subseteq\Fg_{j-1}$. 
If $j\in e$, then $x_j\in\gg_{e,\hg}=\Ker q$, hence $q(x_j)=0\in\Fg_j$. 
If $j\not\in e$, then $x_j\in\Fg_{j-1}+\hg$, hence by $\Fg_{j-1}\subseteq\Fg$ we directly obtain $x_j\in\Fg_{j-1}+(\hg\cap\Fg_j)$. 
Therefore $q(x_j)\in q(\Fg_{j-1})+\hg\cap\Fg_j\subseteq \Fg_{j-1}+\Fg_j=\Fg_j$. 
Using again the induction hypothesis $q(\Fg_{j-1})\subseteq\Fg_{j-1}$ we then obtain $q(\Fg_j)\subseteq\Fg_j$, and we are done. 
\end{example}

\section{Main results on regular orbits}\label{Sect_reg}

In this section we establish some of our main results on regular orbits of linear dynamical systems of nilpotent Lie groups. 
These results include a description of the points where the orbits are regular (Theorem~\ref{ArCuOu16_Th3.6}), a dichotomy for the set of these points (Theorem~\ref{15March2019}), and an interpretation of that dichotomy in terms of transformation groups $C^*$-algebras (Corollary~\ref{main_reg}), which holds if and ony if the group under consideration is abelian (Proposition~\ref{nonabelian}). 

\begin{setting}\label{sett_decomp}
	\normalfont
Throughout this section we keep the following hypothesis and notation.
\begin{itemize}
	\item $G$ is a nilpotent Lie group; 
	\item $\pi\colon G\to\End_{\RR}(\Vc)$ is a continuous representation on the finite-dimensional real vector space~$\Vc$; 
	\item $\Vc=\Vc_1\dotplus\cdots\dotplus\Vc_m$ is a direct sum decomposition into $\pi(G)$-invariant subspaces satisfying the condition that for every $j=1,\dots,m$ there exists a field  $\KK_j\in\{\RR,\CC\}$ and a $\KK_j$-vector space structure of $\Vc_j$ that agrees with its structure of a real vector subspace of~$\Vc$ and there exists a function $\chi_j\colon G\to\KK_j^\times$ for which $\pi(g)\vert_{\Vc_j}\in\End_{\KK_j}(\Vc_j)$ and $(\pi(g)\vert_{\Vc_j}-\chi_j(g)\id_{\Vc_j})^{m_j}=0$ for all $g\in G$, where $m_j:=\dim_{\KK_j}(\Vc_j)$. 
	\item For every $v\in\Vc\setminus\{0\}$ we denote by $\supp\,v$ the set $J\subseteq\{1,\dots,m\}$ for which there exist $v_j\in\Vc_j\setminus\{0\}$ for all $j\in J$ with $v=\sum\limits_{j\in J}v_j$. 
\end{itemize}
\end{setting}

\begin{remark}
\normalfont
The restriction of $\pi$ to its invariant subspace $\Vc_j$ has a joint eigenvector by the classical theorem of S. Lie \cite[Ch. 8, \S 1, Th. 1 and Cor. 2]{BaRa86}, and it then directly follows that the function $\chi_j\colon G\to\KK_j^\times$ is a continuous morphism. 

It then further follows that there exist and are uniquely determined the linear functionals 
$\alpha_j,\beta_j\colon\gg\to\RR$ satisfying $[\gg,\gg]\subseteq(\Ker\alpha_j)\cap(\Ker\beta_j)$ and 
$$(\forall x\in\gg)\quad
 \chi_j(\exp_G x)=\ee^{\alpha_j(x)+\ie\beta_j(x)}.$$
\end{remark}

\begin{definition}\label{sett_decomp_def}
\normalfont
We define the abelian Lie group 
$Z:=\KK_1^\times \times\cdots\times\KK_m^\times$ 
and the mappings 
\begin{align}
E &\colon G\to\End_{\RR}(\Vc),\quad E(g):=\chi_1(g)\id_{\Vc_1}\dotplus\cdots+\chi_m(g)\id_{\Vc_m},
\nonumber\\
E_0 &\colon G\to Z,\qquad\qquad E_0(g):=(\chi_1(g),\dots,\chi_m(g)), \nonumber\\
\nu &\colon G\to\End_{\RR}(\Vc),\quad \nu(g):=E(g)^{-1}\pi(g). \nonumber
\end{align}
For every $v\in\Vc$ we also define $G_\nu(v):=\{g\in G\mid \nu(g)v=v\}$. 
\end{definition}

\begin{remark}\label{enupi}
\normalfont 
We note the following properties of the objects introduced above: 
\begin{enumerate}[{\rm (i)}]
	\item\label{enupi_item1} Both $E$ and $\nu$ are continuous representations with $[E(G),\nu(G)\cup\pi(G)]=\{0\}$ and $\pi(\cdot)=E(\cdot)\nu(\cdot)$. 
	\item\label{enupi_item2} For every $g\in G$ one has $(\nu(g)-\id_{\Vc})^N=0$, where $N=\dim \Vc$, 
	and this further implies that the closed subgroup $G_\nu(v)$ of $G$ is connected for all $v\in\Vc$. 
	(See \cite[Lemma 3.1.1]{CG90}.)
\end{enumerate}
\end{remark}

We now state the following generalization of \cite[Th. 3.6]{ArCuOu16} from abelian to general nilpotent Lie groups.

\begin{theorem}\label{ArCuOu16_Th3.6}
The following assertions are equivalent for every $v\in \Vc\setminus\{0\}$: 
\begin{enumerate}[{\rm (i)}]
	\item\label{ArCuOu16_Th3.6_item1} 
	The subset $\pi(G)v\subseteq \Vc$ is locally compact. 
	\item\label{ArCuOu16_Th3.6_item2} 
	The subgroup $E_0(G)\subseteq Z$ is closed. 
	\item\label{ArCuOu16_Th3.6_item3} 
	The additive subgroup 
	$\langle\{\beta_j\vert_{\gg_\nu(v)}\mid j\in\supp\, v\}\rangle\subseteq\gg_\nu(v)^*$ 
	is discrete. 
\end{enumerate}
\end{theorem}

The proof requires some preparation.

\begin{notation}\label{not_X}
\normalfont
We denote  
$$\begin{aligned}
X:=
& \{v\in\Vc\setminus\{0\}\mid\supp\,v=\{\1,\dots,m\}\} \\
=
&\{v_1+\cdots+v_m\mid v_j\in\Vc_j\setminus\{0\}\text{ for }j=1,\dots,m\}.
\end{aligned}$$
For $j=1,\dots,m$ and $v\in\Vc_j\setminus\{0\}$ we also denote $[v]:=\KK_j v\subseteq\Vc$. 
We also define the projective space $\PP(\Vc_j):=\{[v]\mid v\in\Vc_j\setminus\{0\}\}$, endowed with its usual topology of a compact space for which the mapping $\Vc_j\setminus\{0\}\to\PP(\Vc_j)$, $v\mapsto[v]$, is a quotient mapping. 

We denote $W:=\PP(\Vc_1)\times\cdots\times\PP(\Vc_m)$ and set
$$q\colon X\to W,\quad q(v_1+\cdots+v_m):=([v_1],\dots,[v_m]).$$
Consider the group action 
$$\widetilde{\nu}\colon G\times W\to W,\quad 
\widetilde{\nu}(g,w):=g.w$$
where $g.([v_1],\dots,[v_m]):=([\nu(g)v_1],\dots,[\nu(g)v_m])$ 
for all $w=([v_1],\dots,[v_m])\in W$ and $g\in G$. 
We note for later reference that 
\begin{equation}\label{nupi}
(\forall j\in\{1,\dots,m\})(\forall v_j\in\Vc_j\setminus\{0\})(\forall g\in G)\quad [\pi(g)v_j]=[\nu(g)v_j]
\in\PP(\Vc_j).
\end{equation}
\end{notation}

\begin{lemma}\label{18March2019}
One has the commutative diagram 
$$
\xymatrix
{G\times X \ar[r]^{\pi}	\ar[d]_{\id_G\times q} & X  \ar[d]^{q}   \\
G\times W \ar[r]^{\widetilde{\nu}} &  W
}
$$
and for every $v\in X$ the mapping $q\vert_{\nu(G)v}\colon \nu(G)v\to G.q(v)=\widetilde{\nu}(G\times\{q(v)\})$ is a $G$-equivariant homeomorphism.  
\end{lemma}

\begin{proof}
Commutativity of this diagram follows by \eqref{nupi}. 

Now let $v=v_1+\cdots+v_m\in X$ arbitrary, where $v_j\in\Vc_j\setminus\{0\}$ for $j=1,\dots,m$. 
It is easily seen that the mapping $q\vert_{\nu(G)v}\colon \nu(G)v\to G.q(v)$ is continuous, surjective, and $G$-equivariant. 
To show that it is also injective, let $g_1,g_2\in G$ arbitrary with $q(\nu(g_1)v)=q(\nu(g_2)v)$. 
Then, for $j=1,\dots,m$, there exists $t_j\in\KK_j^\times$ with $\nu(g_1)v_j=t_j\nu(g_2)v_j$, hence $\nu(g_2^{-1}g_1)v_j=t_jv_j$. 
Now $t_j=1$ by Remark~\ref{enupi}\eqref{enupi_item2}, 
and then $\nu(g_1)v=\nu(g_2)v$, which shows that $q\vert_{\nu(G)v}$ is injective. 

It remains to check that the inverse of $q\vert_{\nu(G)v}\colon \nu(G)v\to G.q(v)$ is continuous. 
To this end it suffices to prove that for every sequence  $\{g_n\}_{n\in\NN}$ in $G$ with $\lim\limits_{n\to\infty}\widetilde{\nu}(g_n,q(v))=q(v)$, one has $\lim\limits_{n\to\infty}\nu(g_n)v=v$. 
In fact, for $j=1,\dots,n$ one has $\lim\limits_{n\to\infty}[\nu(g_n)v_j]=[v_j]$ in $\PP(\Vc_j)$, 
hence there exists a sequence $\{t_n\}_{n\in\NN}$ in $\KK_j$ with 
\begin{equation}\label{18March2019_proof_eq1}
\lim\limits_{n\to\infty}t_n\nu(g_n)v_j=v_j
\end{equation} 
in $\Vc_j\setminus\{0\}$.
On the other hand, by S.\ Lie's theorem for the unipotent representation $G\to\End(\Vc_j)$, $g\mapsto\nu(g)\vert_{\Vc_j}$, 
there exists a basis in the $\KK_j$-vector space $\Vc_j$ with respect to which one has a lower-triangular matrix representation
$$\nu(\cdot)\vert_{\Vc_j}=
\begin{pmatrix}
1 & & 0 \\
  &\ddots & \\
\ast & & 1
\end{pmatrix}
$$
and it then easily follows by \eqref{18March2019_proof_eq1} that 
$\lim\limits_{n\to\infty}t_n=1$ in $\KK_j^\times$, hence 
one obtains by~\eqref{18March2019_proof_eq1} again that 
$\lim\limits_{n\to\infty}\nu(g_n)v_j=v_j$ for every $j$.
This further implies $\lim\limits_{n\to\infty}\nu(g_n)v=v$, 
which completes the proof of the fact that the inverse of  $q\vert_{\nu(G)v}\colon \nu(G)v\to G.q(v)$ is continuous. 
This concludes the proof.
\end{proof}

\begin{lemma}\label{lattices}
For any connected closed subgroup $H\subseteq G$ with its Lie algebra $\hg\subseteq \gg$ define $B_H:=\{\beta_j\vert_\hg\mid j=1,\dots,m\}\subseteq\hg^*$. 
The following conditions are equivalent: 
\begin{enumerate}[{\rm(i)}]
	\item\label{lattices_item1} 
	The subgroup $E_0(H)\subseteq Z$ is closed. 
	\item\label{lattices_item2} One has $\dim_{\RR}(\spa_{\RR}(B_H))=\dim_{\QQ}(\spa_{\QQ}(B_H))$. 
	\item\label{lattices_item3} The additive subgroup $\langle B_H\rangle\subseteq\hg^*$ 
	is discrete. 
\end{enumerate}
\end{lemma}

\begin{proof} 
It is easily seen that Assertions \eqref{lattices_item2}--\eqref{lattices_item3} are equivalent by Lemma~\ref{dic0}. 
	
Let $(H,H)$ denote the subgroup of $H$ generated by its subset $\{hgh^{-1}g^{-1}\mid h,g\in H\}$. 
Since $H$ is a nilpotent Lie group,  $(H,H)$ is the closed connected subgroup of $H$ whose Lie algebra is $[\hg,\hg]$. 
(See \cite[Ch. XII, Th. 3.1]{Ho65}.)
Moreover, $[\hg,\hg]\subseteq\Ker\beta_j$ for $j=1,\dots,m$, and on the other hand $(H,H)\subseteq \Ker E_0$ since $E_0\colon G\to Z$ and the group $Z$ is abelian. 

Now define the abelian Lie group $\widetilde{H}:=H/(H,H)$ with its Lie algebra $\widetilde{\hg}:=\hg/[\hg,\hg]$, 
and denote by $P\colon H\to\widetilde{H}$ and $p\colon\hg\to\widetilde{\hg}$ their corresponding quotient maps. 
It follows by the above observations that there exist and are uniquely determined the 
continuous group morphisms $\widetilde{\chi}_j\colon\widetilde{H}\to \KK_j^\times$ and linear functionals $\widetilde{\beta}_j\colon\widetilde{\hg}\to\RR$ 
satisfying $\widetilde{\chi}_j\circ P=\chi_j\vert_H$ and $\widetilde{\beta}_j\circ p=\beta_j\vert_{\hg}$ for $j=1,\dots,m$. 
Defining $\widetilde{E}_0:=(\widetilde{\chi}_1,\dots,\widetilde{\chi}_m)\colon \widetilde{H}\to Z$, one also has $\widetilde{E}_0\circ P=E$. 
If we also denote 
$$\widetilde{B_H}:=\{\widetilde{\beta}_j\mid j=1,\dots,m\}\subseteq\widetilde{\hg}^*$$ 
then  Assertions \eqref{lattices_item1}, \eqref{lattices_item2}, and \eqref{lattices_item3} from the statement are respectively equivalent to the following assertions: 
\begin{enumerate}[{\rm(i')}]
	\item\label{lattices_proof_item1} 
	The subgroup $\widetilde{E}_0(\widetilde{H})\subseteq Z$ is closed. 
	\item\label{lattices_proof_item2} 
	One has $\dim_{\RR}(\spa_{\RR}(\widetilde{B_H}))=\dim_{\QQ}(\spa_{\QQ}(\widetilde{B_H}))$. 
	\item\label{lattices_proof_item3} 
	The additive subgroup $\langle \widetilde{B_H}\rangle=\ZZ\widetilde{\beta}_1+\cdots+\ZZ\widetilde{\beta}_m\subseteq\widetilde{\hg}^*$ is discrete. 
\end{enumerate}
Here $\widetilde{H}$ is an abelian (simply connected) Lie group, hence one easily obtains that Assertions  \eqref{lattices_proof_item1}--\eqref{lattices_proof_item3} are equivalent by \cite[Ch. VII, \S 1, no. 5]{Bo74}. 
Moreover, as we already noted at the beginning of the present proof, Assertions \eqref{lattices_proof_item2}--\eqref{lattices_proof_item3} are equivalent by Lemma~\ref{dic0}. 

Therefore Assertions \eqref{lattices_item1}--\eqref{lattices_item3} are equivalent. 
\end{proof}

\begin{proof}[Proof of Theorem~\ref{ArCuOu16_Th3.6}]
Let $v\in \Vc\setminus\{0\}$. 
Restricting the representation $\pi\colon G\to\End_{\RR}(\Vc)$ 
to its invariant subspace $\bigoplus\limits_{j\in\supp\,v}\Vc_j$, 
we may assume $v\in X$, hence 
$v=v_1+\cdots+v_m\in X$ with $v_j\in\Vc_j\setminus\{0\}$ for $j=1,\dots,m$.	
Denote $w:=q(v)\in W$ and its corresponding isotropy group $G_{\widetilde{\nu}}(w)=\{g\in G\mid \widetilde{\nu}(g,w)=w\}$. 
We actually have that
\begin{equation}\label{ArCuOu16_Th3.6_proof_eq1}
G_{\widetilde{\nu}}(w)=G_\nu(v).
\end{equation} 
In fact, for arbitrary $g\in G$, 
$$\nu(g)v=v\iff q(\nu(g)v)=q(v)\iff \widetilde{\nu}(g,q(v))=q(v)
\iff g\in G_{\widetilde{\nu}}(w)$$
since  the mapping $q\vert_{\nu(G)v}$ is injective by Lemma~\ref{18March2019} and that the diagram from that lemma is commutative. 

Recall from Remark~\ref{enupi}\eqref{enupi_item2} that $\nu\colon G\to\End(\Vc)$ is a unipotent representation, hence $\nu(G)v$ is a locally compact subset of $\Vc$. 
(See for instance \cite[Th. 3.1.4]{CG90}.) 
This implies by Lemma~\ref{18March2019} that the subset $G.w\subseteq W$ is locally compact.  
Using \eqref{ArCuOu16_Th3.6_proof_eq1} and the commutative diagram from Lemma~\ref{18March2019}, it then follows by Lemma~\ref{ACDO3.1_gen}\eqref{ACDO3.1_gen_item4} that one has 
$$\pi(G)v\text{ locally compact }\subseteq \Vc
\iff  
\pi(G_\nu(v))v\text{ locally compact }\subseteq\Vc.$$ 
Thus, to complete the proof that Assertions \eqref{ArCuOu16_Th3.6_item1}--\eqref{ArCuOu16_Th3.6_item2} are equivalent, it suffices to show the following: 
\begin{equation}\label{ArCuOu16_Th3.6_proof_eq2}
\pi(G_\nu(v))v\text{ locally compact }\subseteq\Vc\iff 
E_0(G)\text{ closed }\subseteq Z.
\end{equation}
To this end we recall that 
$$G_\nu(v)=\{g\in G\mid\nu(g)v=v\}=\{g\in G\mid\nu(g)v_j=v_j\text{ for }j=1,\dots,m\}$$
hence  
$$(\forall g\in G_\nu(v))\quad 
\pi(g)v=E(g)\nu(g)v
=\sum_{j=1}^m E(g)\nu(g)v_j
=\sum_{j=1}^m E(g)v_j
=\sum_{j=1}^m \chi_j(g)v_j.
$$
This easily implies that the mapping 
$$\pi(G_\nu(v))v\to E_0(G_\nu(v)),\quad 
\pi(g)v\mapsto (\chi_1(g),\dots,\chi_m(g))=E_0(g)$$
is a homeomorphism, and then $\pi(G_\nu(v))v$ is a locally compact subset of $\Vc$ if and only if $E_0(G_\nu(v))$ is a locally compact subset of $Z$. 
This is further equivalent to the condition that $E_0(G_\nu(v))$ be a locally closed subset of $Z$ by \cite[Lemma 1.26]{Wi07}. 
Since $E_0(G_\nu(v))$ is actually a subgroup of $Z$, 
and any subgroup of a topological group is locally closed if and only if it is closed (by \cite[Ch. I, Prop. 2.1]{Ho65}), the proof of \eqref{ArCuOu16_Th3.6_proof_eq2} is complete. 

Assertions \eqref{ArCuOu16_Th3.6_item2}--\eqref{ArCuOu16_Th3.6_item3} are equivalent by Lemma~\ref{lattices}. This completes the proof. 
\end{proof}

Theorem~\ref{15March2019} below is a generalization of \cite[Cor. 5.6]{ArCuOu16} 
from abelian to general nilpotent Lie groups. 
We use the following notation. 

\begin{notation}\label{pre_15March2019}
\normalfont 
Let $k_0:=\min\limits_{v\in\Vc}\dim\gg_\nu(v)$, 
  $\Vc_\gen:=\{v\in\Vc\mid\dim\gg_\nu(v)=k_0\}$.
Then the set $\Vc_\gen$ is an open dense $\pi(G)$-invariant subset of $\Vc$ and one has $\Vc_\gen\subseteq X$.
We also consider the set 
$$\Gamma:=\{v\in\Vc\mid \pi(G)v\text{ locally compact }\subseteq\Vc\}$$
of the regular points  in $\Vc$.
\end{notation}

\begin{theorem}\label{15March2019}
	Let $G$ be a nilpotent Lie group and 
	$\pi\colon G\to\End_{\RR}(\Vc)$  a continuous representation on the finite-dimensional real vector space~$\Vc$. 
If $\mathfrak{k}\subseteq\gg$ is a linear subspace  
for which the set 
\begin{equation}\label{15March2019_eq1}
\Vc_\kg:=\{v\in\Vc\mid\gg_\nu(v)\dotplus\kg=\gg\}
\end{equation}
is a dense open subset of $\Vc_\gen$   
then exactly one of the following cases occurs: 
\begin{enumerate}[{\rm(i)}]
	\item\label{15March2019_item1} The set $\Int(\Vc_\kg\cap\Gamma)$ is a dense open subset of $\Vc_\kg$.  
	\item\label{15March2019_item2} The set $\Vc_\kg\setminus\Gamma$ is  dense in $\Vc_\kg$, and either $\Int(\Vc_\kg\setminus\Gamma)=\emptyset$ or $\Int(\Vc_\kg\setminus\Gamma)$ is dense in~$\Vc_\kg$.   
\end{enumerate}
\end{theorem}

\begin{proof}
If $\Vc_\kg\cap\Gamma=\emptyset$, then \eqref{15March2019_item2} clearly holds true. 

Now let us assume $\Vc_\kg\cap\Gamma\ne\emptyset$ and select $v_0\in\Vc_\kg\cap\Gamma$. 
The proof proceeds in several steps. 
Recall that for every linear subspace $\hg\in\Gr_\kg(\gg)$ we consider $E_\kg(\hg)\colon \gg\to\gg$ its corresponding oblique projection onto~$\hg$, 
corresponding to the direct sum decomposition $\gg=\hg\dotplus\kg$. 

Step~1. The function $\Psi\colon\Vc_\kg\to\Bc(\gg)$, $\Psi(v):=E_\kg(\gg_\nu(v))$ is rational. 

This follows by Proposition~\ref{norm0}, for $T\colon \Vc_\kg\to\Bc(\gg,\Vc)$, 
$T(v):=\de\nu(\cdot)v$, which satisfies $\Ker T(v)=\gg_\nu(v)\in\Gr_\kg(\gg)$ for all $v\in\Vc_\kg$. 

Step~2. For every $j=1,\dots,m$ the function 
$$f_j\colon \Vc_\kg\to\gg_\nu(v_0)^*,\quad 
f_j(v):=\beta_j\circ \Psi(v)\vert_{\gg_\nu(v_0)}$$ 
is rational and one has 
\begin{equation}\label{15March2019_proof_eq1}
\Vc_\kg\cap\Gamma=\{v\in\Vc_\kg\mid \langle f_1(v),\dots,f_m(v)\rangle \text{ discrete }\subseteq\gg_\nu(v_0)^*\}.
\end{equation}
In fact, it directly follows by Step~1 that the functions $f_1,\dots,f_m\colon \Vc_\kg\to\gg_\nu(v_0)^*$ are rational. 
Moreover, for arbitrary $v\in\Vc_\gen$ one has $\kg\cap \gg_\nu(v_0)=\{0\}$. 
It is then straightforward to check that the linear operator 
$$\Psi(v)\vert_{\gg_\nu(v_0)}\colon\gg_\nu(v_0)\to\gg_\nu(v)$$
is injective hence invertible since $\dim\gg_\nu(v)=\dim\gg_\nu(v_0)$. 
One then obtains for arbitrary $v\in\Vc_\kg$ 
\begin{align}\langle \beta_1\vert_{\gg_\nu(v)},\dots, & \beta_m\vert_{\gg_\nu(v)}\rangle\text{ discrete }\subseteq\gg_\nu(v)^*
\nonumber \\
& \iff 
\langle \beta_1\circ \Psi(v)\vert_{\gg_\nu(v_0)},\dots,\beta_m\circ \Psi(v)\vert_{\gg_\nu(v_0)}\rangle\text{ discrete }\subseteq\gg_\nu(v_0)^* \nonumber \\
& \iff 
\langle f_1(v),\dots,f_m(v)\rangle\text{ discrete }\subseteq\gg_\nu(v_0)^* \nonumber 
\end{align}
and now \eqref{15March2019_proof_eq1} follows by Theorem~\ref{ArCuOu16_Th3.6} 
since $\Vc_\kg\subseteq\Vc_\gen\subseteq X$. 
(See Notation~\ref{pre_15March2019} and~\ref{not_X}.)

Step~3. 
The above Step~2 shows that Lemma~\ref{dic3} is applicable with $A=\Vc_\kg\cap\Gamma$, and we then obtain that either $\Vc_\kg\cap\Gamma$ is an open dense subset of $\Vc_\kg$, or $\Vc_\kg\setminus\Gamma$ ($=\Vc_\kg\setminus A$) is a dense subset of $\Vc_\kg$.   
This completes the proof.  
\end{proof}

\begin{corollary}\label{main_reg}
In the conditions and notation of Theorem~\ref{15March2019},
exactly one of the following cases can occur: 
	\begin{enumerate}[{\rm(i)}]
		\item The set $\Int \Gamma$ is dense in $\Vc$. 
		\item The set $\Vc\setminus \Gamma$ is dense in~$\Vc$, 
		and either $\Int(\Vc\setminus\Gamma)=\emptyset$ or $\Int(\Vc\setminus\Gamma)$ is dense in~$\Vc$. 
	\end{enumerate} 
\end{corollary}

\begin{proof}
It follows by Example~\ref{strat}\eqref{strat_item2} and \eqref{strat_eq1} that 
there exists a linear subspace $\kg\subseteq\gg$ with $\dim(\gg/\kg)=k_0$ 
for which the set 
$\Vc_\kg$ defined in~\eqref{15March2019_eq1} 
is a dense open subset of $\Vc_\gen$.  
Theorem~\ref{15March2019} is then applicable, and one directly obtains the assertion, since 
$\Vc_\gen$ is in turn a dense open subset of~$\Vc$. 
\end{proof}

\begin{remark}\label{correm}
	\normalfont
	The above corollary implies that $\Int\Gamma$ is not dense in $\Vc$ if and only if $\Int\Gamma=\emptyset$.
\end{remark}

\begin{corollary}\label{27August2019}
	If the group $G$ is abelian, then exactly one of the following cases can occur: 
	\begin{enumerate}[{\rm(i)}]
		\item The set $\Int \Gamma$ is dense in $\Vc$. 
		\item The transformation-group $C^*$-algebra $G\ltimes \Cc_0(\Vc)$ is antiliminary. 
	\end{enumerate} 
\end{corollary}

\begin{proof}
The transformation-group $C^*$-algebra $G\ltimes \Cc_0(\Vc)$ is not antiliminary if and only if it has at least one nonzero postliminary ideal, 
that is, if and only if its largest postliminary ideal is nonzero. 
Since the group $G$ is abelian, the largest postliminary ideal of  $G\ltimes \Cc_0(\Vc)$ is $G\ltimes \Cc_0(\Int\Gamma)$  by Lemma~\ref{anHuefWilliams}.
Therefore $G\ltimes \Cc_0(\Vc)$ is not antiliminary if and only if $G\ltimes \Cc_0(\Int\Gamma)\ne\{0\}$, which is further equivalent to $\Int\Gamma\ne\emptyset$. 
Finally, by Corollary~\ref{main_reg}, one has $\Int\Gamma\ne\emptyset$ if and only if $\Int\Gamma$ is dense in~$\Vc$. 
\end{proof}

The assertions of  Corollary~\ref{27August2019} do not carry over directly when $G$ is a noncommutative nilpotent Lie group,  as seen from the following proposition. 

\begin{proposition}\label{nonabelian}
Let $G$ be a nilpotent Lie group. 
If $G$ is not commutative, then there exists a continuous representation $\pi\colon G\to\End_\CC(\CC^2)$ for which $\Int\Gamma=\emptyset$ 
while the $C^*$-algebra $G\ltimes_\pi\Cc_0(\CC^2)$ is neither antiliminary nor postliminary. 
\end{proposition}

\begin{proof}
Let $\gg$ be the Lie algebra of $G$ and $\zg$ be the center of $\gg$.  
Since $G$ is simply connected, we may assume $G=(\gg,\cdot)$, where the group operation $\cdot$ is given by the Baker-Campbell-Hausdorff formula. 
The Lie algebra $\gg$ is nilpotent and $\gg^{(1)}:=[\gg,\gg]\ne\{0\}$, 
and it then easily follows that 
$\gg^{(1)}+\zg\subsetneqq\gg$, 
hence we may select $X_1\in\gg\setminus(\gg^{(1)}+\zg)$,
and a linear subspace $\gg_0\subseteq\gg$ with 
$\gg^{(1)}+\zg\subseteq\gg_0$
and $\gg=\RR X_1\dotplus\gg_0$. 
Then $\gg_0$ is an ideal of $\gg$.  
In addition, since $\zg\subseteq\gg_0$ and $X_0\not\in\gg_0$, we have  $X_0\not\in\zg$, 
hence $[X_1,\gg_0]\ne\{0\}$. 

Now $G_0:=(\gg_0,\cdot)$ is a closed normal subgroup of $G$ and 
we have the semidirect product decomposition $G=G_0 \rtimes \RR X_1$, 
using the diffeomorphism 
 $\gg_0\times \RR \to \gg$, $(Y, t)\mapsto  Y \cdot tX_1$. 
We further select any $\theta\in\RR\setminus\QQ$ and we show that 
$$\pi\colon G\to\End_\CC(\CC^2),\quad 
\pi( Y \cdot tX_1):=\begin{pmatrix}
\ee^{\ie t} & 0 \\0 & \ee^{\ie \theta t}
\end{pmatrix}
$$
is a representation with the stated properties. 
Since $G_0$ is a normal subgroup of $G$, it is clear that $\pi$ is a group representation. 
Moreover, since $\theta\in\RR\setminus\QQ$, 
it is easily seen that $\Gamma=\Bigl\{ \begin{pmatrix}
z_1 \\ z_2
\end{pmatrix}\in\CC^2 \mid z_1 z_2=0\Bigr\}$ 
hence $\Int\Gamma=\emptyset$. 
Moreover, defining 
$$\rho\colon \RR X_1\to \End_\RR(\gg_0\dotplus\CC^2),\quad 
\rho(tX_1):=\begin{pmatrix}
\exp(t(\ad_{\gg}X_1)\vert_{\gg_0}) & 0 \\0 & \pi(tX_1)
\end{pmatrix}
$$
and using the semidirect product decomposition $G= G_0\rtimes \RR X_1$, 
we obtain a natural $*$-isomorphism 
$G\ltimes_\pi\Cc_0(\CC^2)\simeq \RR X_1\ltimes_\rho\Cc_0(\gg_0\dotplus\CC^2)$ since $G_0$ acts trivially on $\CC^2$ via~$\pi$. 
(See for instance \cite[Prop. 3.11]{Wi07}.) 
If we denote by $\Gamma_\rho$ the set of all points in $\gg_0\dotplus\CC^2$ whose corresponding orbits via the action~$\rho$ are regular, 
then we will show that $\Int\Gamma_\rho$ is a dense subset of $\gg_0\dotplus\CC^2$ while $\Gamma_\rho\ne\gg_0\dotplus\CC^2$, 
and this will imply that the $C^*$-algebra  $\RR X_1\ltimes_\rho\Cc_0(\gg_0\dotplus\CC^2)$ is neither postliminary 
(by \cite[Thms. 6.2 and 8.43]{Wi07}) nor antiliminary 
(by Corollary~\ref{27August2019}).  

In fact $\Gamma_\rho\cap\CC^2=\Gamma\ne\CC^2$, hence $\Gamma_\rho\ne\gg_0\dotplus\CC^2$. 
On the other hand, 
the nilpotent operator $(\ad_{\gg}X_1)\vert_{\gg_0}$ is different from zero since $[X_1,\gg_0]\ne\{0\}$. 
Using the well-known formula for the exponential of a nilpotent Jordan cell,  
it follows that
if $Y\in\gg_0\setminus\Ker(\ad_{\gg}X_1)$ 
and $z\in\CC^2$ 
then 
$\lim\limits_{t\to\pm\infty}\rho(tX_1)\begin{pmatrix}
Y \\ z
\end{pmatrix}=\infty$ 
hence 
$\begin{pmatrix}
Y \\ z
\end{pmatrix}\in\Gamma_\rho$ 
by the last assertion of Lemma~\ref{ACDO2.3_new}. 
This shows that  $\Int\Gamma_\rho$ is a dense subset of $\gg_0\dotplus\CC^2$. 
\end{proof}

\begin{remark}\label{main_reg_ab}
\normalfont 
As in Setting~\ref{sett_decomp}, let $G$ be a nilpotent Lie group with  a continuous finite-dimensional representation $\pi\colon G\to\End_{\RR}(\Vc)$. 
One can then form the semidirect product of Lie groups $R:=\Vc\rtimes_\pi G$, which is a solvable Lie group with its Lie algebra  $\rg:=\Vc\rtimes_{\de\pi}\gg$ obtained as a semidirect product of Lie algebras, 
however the relation between the coadjoint orbits of $S$ and the $G$-orbits in $\Vc$ is rather involved especially if the group $G$ is nonabelian, as discussed in \cite{Ra75} and \cite{Ba98}.  
Specifically, the dual of the Lie algebra of $R$ is $\rg^*=\Vc^*\times\gg^*$, 
and the coadjoint action $\Ad_R^*\colon R\times\rg^*\to\rg^*$ is given by the formula 
\begin{equation}
\label{main_reg_ab_eq1}
(\Ad_R^*(v,g))(p,\xi)=(\pi(g)^*p,\Ad_G^*(g)\xi-\theta_p(v))
\end{equation}
for all $v\in \Vc$, $g\in G$, $p\in\Vc^*$, and $\xi\in\gg^*$, 
where $\theta_p\colon \Vc\to\gg^*$, $v\mapsto \langle\de\pi(\cdot)^*p,v\rangle$ and 
$\langle\cdot,\cdot\rangle\colon\Vc^*\times\Vc\to\RR$ is the natural duality pairing. 

Using these remarks, a version of Corollary~\ref{main_reg} 
\emph{in the special case when $G$ is abelian} studied in \cite{ArCuDaOu13} and \cite{ArCuOu16} can be alternatively obtained as follows. 
Namely, if $G$ is abelian, then $\Ad_G^*(g)\xi=\xi$ for all $g\in G$ and $\xi\in\gg^*$, hence, by \eqref{main_reg_ab_eq1}, 
\begin{equation}
\label{main_reg_ab_eq2}
(\Ad_R^*(R))(p,\xi)=\pi(G)^*p\times (\xi+\theta_p(\Vc))\subseteq\Vc^*\times\gg^* 
\text{ for all }p\in\Vc^*,\ \xi\in\gg^*.
\end{equation}
Here $\xi+\theta_p(\Vc)$ is an affine subspace, hence a closed subset 
of~$\gg^*$. 
On the other hand, it is straightforward to check that for any topological spaces $X$ and $Y$, a subset $A\subseteq X$ is locally closed if and only if $A\times Y\subseteq X\times Y$ is locally closed. 
Therefore~\eqref{main_reg_ab_eq2} implies that, for any $p\in\Vc^*$ and  $\xi\in\gg^*$, 
the $G$-orbit $\pi(G)^*p\subseteq \Vc^*$ is locally closed if and only if the coadjoint $S$-orbit $(\Ad_R^*(R))(p,\xi)\subseteq\rg^*$ is locally closed. 
On the other hand, since $R$ is a connected solvable Lie group, it follows by \cite[Ch. IV, Prop. 8.2]{Pu71} that, denoting by $E_c$ the set of all points of $\rg^*$ whose coadjoint orbits are locally closed, 
then either $E_c$ or its complement $\rg^*\setminus E_c$ have Lebesgue measure zero. 
Therefore, by the above remarks, either the set $\Gamma^*$ of points in $\Vc^*$ whose $G$-orbits are locally closed, or its complement $\Vc^*\setminus\Gamma^*$, have Lebesgue measure zero. 
This is the aforementioned version of Corollary~\ref{main_reg} for the representation $\pi^*\colon G\to\End(\Vc^*)$, $g\mapsto\pi(g^{-1})^*$ when $G$ is abelian. 

Moreover, by \eqref{main_reg_ab_eq2} again, all $G$-orbits in $\Vc^*$ are locally closed if and only if all coadjoint $R$-orbits are locally closed. 
If this is the case, then one also has: 
\begin{enumerate}[{1.}]
	\item The set $\widehat{R}\setminus\widehat{G}$ is a dense open subset of the unitary dual space $\widehat{G}$ by \cite[Lemma 4.5]{ArKS12}. 
	(We note that $G\simeq R/\Vc$, hence $\widehat{G}$ is always a closed subset of $\widehat{R}$.)
	\item The solvable Lie group $R$ is type~I by \cite[Th. V.3.4]{AuKo71}, 
	since the coadjoint $R$-orbit symplectic forms are exact, for instance by \cite[Prop. 4.4(1.)]{Ba98}. 
	Compare \cite[Prop. 4.3]{ArCuDa12} and Theorem~\ref{post-anti}\eqref{post-anti_item1} below. 
\end{enumerate}
\end{remark}

The following example shows that Corollary~\ref{27August2019} may not carry over 
even in the case of a noncommutative nilpotent Lie group that acts on $\CC^2$ in such a manner that its corresponding orbits are regular and the coadjoint orbits of the corresponding semidirect product group are locally closed. 

\begin{example}\label{Dixmier}
	\normalfont
	We consider the 3-dimensional Heisenberg group 
	$$H_3=\Bigl\{\begin{pmatrix}
	1 & a & c \\
	0 & 1 & b \\
	0 & 0 & 1
	\end{pmatrix}
	\mid a,b,c\in\RR\Bigr\}$$
	and its representation 
	$$\pi\colon H_3\to\End_\CC(\CC^2),\quad 
	\pi\begin{pmatrix}
	1 & a & c \\
	0 & 1 & b \\
	0 & 0 & 1
	\end{pmatrix} 
	=\begin{pmatrix}
	\ee^{\ie a} & 0 \\
	0 & \ee^{\ie b}
	\end{pmatrix}.
	$$
	Then the orbit of every point of $\CC^2$ is a torus of dimension 0, 1, or 2, hence $\Gamma=\CC^2$. 
	Neverteless, as we show below, the $C^*$-algebra $H_3\ltimes\CC_0(\CC^2)$ is antiliminary. 
	
	In fact $H_3\ltimes\CC_0(\CC^2)$ is $*$-isomorphic to the $C^*$-algebra of the Dixmier group $S:=\CC^2\rtimes H_3$ of \cite{Di61}. 
	If $\sg$ is the Lie algebra of $S$, 
	let $S(\xi)$ be the coadjoint isotropy group of any point $\xi\in\sg^*$, 
	$\sg(\xi)$ be the Lie algebra of $S(\xi)$, 
	and $S(\xi)_\1$ be the connected component of $S(\xi)$ with $\1\in S(\xi)_\1$. 
	Then it is known that there exists a unique Lie group morphism $\chi_\xi\colon S(\xi)_\1\to\TT$ whose derivative at $\1$ is $\ie\xi\vert_{\sg(\xi)}$, 
	and we introduce as in \cite[Def. 4.1]{Pu71} the reduced stabilizer of $\xi$, 
	which is the closed subgroup of $S(\xi)$ defined by 
	$$\overline{S(\xi)}:=\{x\in S(\xi)\mid xyx^{-1}y^{-1}\in \Ker\chi_\xi\text{ for all }y\in S(\xi)\}.$$
	We then denote 
	$$E_\infty:=\{\xi\in\sg^*\mid \text{the index of the subgroup } 
	\overline{S(\xi)}\text{ of }S(x)\text{ is infinite}\}.$$ 
	A well-known property of the Dixmier group is that the set $\sg^*\setminus E_\infty$ has its Lebesgue measure in $\sg^*$ equal to zero. 
	See for instance \cite[Ex. 4.7.2]{AC20} for a more precise result in this connection, due to M.~Duflo. 
	It then follows by \cite[Lemma 9.2]{Pu71}
	that if the von Neumann algebra $\textbf{L}(S)$ generated by the regular representation of $S$ is canonically decomposed into $w^*$-closed ideals of types I, II, and III, then 
	the corresponding type I component is equal to $\{0\}$.   
	On the other hand, the von Neumann algebra generated by the regular representation of any simply connected solvable Lie group coincides with its type~I or with its type~II component by \cite[Th. 5]{Pu71}. 
	It then follows that the von Neumann algebra $\textbf{L}(S)$ coincides with its type II component. 
	Then the regular representation of $\CC^2\rtimes H_3$ gives a faithful type II representation of $C^*(\CC^2\rtimes H_3)$, 
	and this implies that $C^*(\CC^2\rtimes H_3)$ is antiliminary. 
	(See for instance \cite[9.5.4]{Di69}.)
\end{example}

\section{Application to generalized $ax+b$-groups}\label{Sect_ax+b}

A generalized $ax+b$-group is a connected simply connected solvable Lie group having a 1-codimensional abelian ideal. 
See also Definition~\ref{def_ax+b} for an alternative description of these groups. 
We study the type~I ideals of the $C^*$-algebra of such a group, 
and our main result (Theorem~\ref{post-anti}) is a precise characterization of the generalized $ax+b$-groups whose $C^*$-algebra is antiliminary, that is, no closed 2-sided ideal is type~I. 
We also discuss the representation theoretic significance of this property (Remark~\ref{III}). 

We recall that a separable $C^*$-algebra $\Ac$ is 
type~I or postliminary if for every irreducible $*$-representation $\pi\colon \Ac\to\Bc(\Hc)$ one has $\Kc(\Hc)\subseteq\pi(\Ac)$. 
The $C^*$-algebra $\Ac$ is called antiliminary if it has no postliminary closed 2-sided ideal different from~$\{0\}$. 
(See~\cite{Di69}.)

In what follows in this section, unless otherwise mentioned, 
$\Vc$ is a finite-dim\-ens\-ional real vector space and 
$D\in \End(\Vc)$. 
We recall the following definitions. (See \cite[\S 1.4]{CW14}, and \cite{BB18} for the relevance  of these notions in the context of $C^*$-algebraic properties of $ax+b$-groups.)
We consider the $\CC$-linear extension $D\colon\Vc_{\CC}\to\Vc_{\CC}$, 
and  let $\sigma(D)$ be its spectrum.
For  $\lambda\in \sigma(D)\cap(\CC \setminus \RR)$, 
the real generalized eigenspace for $\lambda$, $\overline{\lambda}$ is the linear subspace of $\Vc$ given by 
$$ E^D(\lambda):=\{v_1, v_2\in \Vc\mid v_1+ \ie v_2 \in \Ker(D-\lambda \1)^m\},$$
where $m$ is the dimension of the largest Jordan block for $\lambda$, while 
if  $\lambda\in \sigma(D)\cap \RR$, the real generalized eigenspace for $\lambda$ is $E^D(\lambda):=\Ker(D-\lambda\1)^m$, where $m$ is again the dimension of the largest Jordan block for $\lambda$. 
For $\lambda\in\CC\setminus\sigma(D)$,  we set $E^D(\lambda):=\{0\}$.
Then we define 
$\Vc_{\pm}:=\bigoplus\limits_{\pm\Re\lambda>0}E^D(\lambda)$, 
$\Vc_{0}:=\bigoplus\limits_{\Re\lambda= 0}E^D(\lambda)$.
It follows that
$\Vc=\Vc_{-}\oplus\Vc_0\oplus\Vc_{+}$
 Define
$$ D_0:= D\vert_{\Vc_0}.$$

We start with the following lemma, included here for the sake of completeness.

\begin{lemma}\label{Diophant}
	For any set $S\subseteq\RR$ the subgroup of $(\RR,+)$ generated by $S$ is closed if and only if there exists $\theta\in\RR\setminus\{0\}$ with $S\subseteq\{r\theta\mid r\in\QQ\}$, which is further equivalent to $\dim_{\QQ}(\spa_{\QQ}(S))=1$. 
\end{lemma}

\begin{proof}
	This follows from the well-known fact that for any $\theta_1,\theta_2\in\RR\setminus\{0\}$ the set $\{m_1\theta_1+m_2\theta_2\mid m_1,m_2\in\ZZ\}$ is dense in $\RR$ if and only if $\theta_1/\theta_2\in\RR\setminus\QQ$. 
\end{proof}

\begin{notation}\label{sett_post-anti}
	\normalfont 
	For any $D\in\End(\Vc)$, we denote by $S_D:=\langle\ie \sigma(D)\cap\RR\rangle$ the subgroup of $(\RR,+)$ generated by the imaginary parts of the purely imaginary eigenvalues of~$D$.  
\end{notation}

Some of the assertions of Lemma~\ref{elliptic} below are stated without proof in \cite{ArCuDaOu13}.
Recall that $\Gamma$ is the set of $v\in\Vc$ 
for which the image of the map $\gamma_v\colon \RR\to\Vc$, $t\mapsto\ee^{tD}v$, is a locally closed subset of~$\Vc$.

\begin{lemma}\label{elliptic}
	Let $\Vc$ be a finite-dimensional real vector space, 
	$D\in \End(\Vc)$. 
	
{\rm(a)}  The following conditions are equivalent.
	\begin{enumerate}[{\rm(i)}]
	\item\label{elliptic_item1} $\Gamma= \Vc$.
	\item\label{elliptic_item2} 
		$S_D$ is closed.
			 \end{enumerate}	
			 
	{\rm(b)} 
		If the set $S_D$ is not closed, then $\Int\Gamma$ is not dense in $\Vc$ if and only if $\sigma(D)\subseteq \ie \RR$ and $D$ is semisimple. 
	\end{lemma}

\begin{proof}
 By \cite[Lemma 2.1 and proof of Th. 2.2]{BB18} we may assume without loss of generality 
\begin{equation}\label{dyn_proof_eq11}
D=
\begin{pmatrix}
-\id_{\Vc_{-}} & 0 & 0 \\
0 & D_0 & 0 \\
0 & 0 & \id_{\Vc_{+}}
\end{pmatrix}
\end{equation}
with respect to the direct sum decomposition $\Vc=\Vc_{-}\oplus\Vc_0\oplus\Vc_{+}$. 	
For every $v\in\Vc$ we denote by $v=v_-+v_0+v_+$ its decomposition with $v_0\in\Vc_0$ and $v_\pm\in\Vc_\pm$.

If $\Vc \ne \Vc_0$ and $v\in\Vc\setminus\Vc_0$, then either $v_-\ne0$ or $v_+\ne0$, hence by \eqref{dyn_proof_eq11} the mapping $\RR\to\Vc$, $t\mapsto \ee^{tD}v$, is injective and then Lemma~\ref{ACDO2.3_new}\eqref{ACDO2.3_new_item4} implies that the orbit $\ee^{\RR D}v$ is locally closed in $\Vc$. 
Indeed, if for instance $v_-\ne0$, $\lim\limits_{n\in\NN}\vert t_n\vert=\infty$, and $\lim\limits_{n\in\NN}\ee^{t_n D}v=v$, 
then one has $v_-=\lim\limits_{n\in\NN}\ee^{t_n D}v_-
=\lim\limits_{n\in\NN}\ee^{t_n}v_-$. 
Since $v_-\ne0$, this is impossible since $\lim\limits_{n\in\NN}\vert t_n\vert=\infty$ and at least one of the sets $\{n\in\NN\mid t_n>0\}$ and $\{n\in\NN\mid t_n<0\}$ is infinite.
It follows that $\Vc\setminus \Vc_0 \subseteq \Gamma$.

(a) \eqref{elliptic_item1}$\Rightarrow$\eqref{elliptic_item2}
It is enough to show that there is $\theta\in \RR$ such that $\ie\sigma(D) 
\cap \RR \subseteq \theta \QQ$.
To this end we prove that for any $\theta_1,\theta_2\in\RR\setminus\{0\}$ are such that $\ie\theta_1,\ie\theta_2\in\sigma(D)\cap\ie\RR$, then
	$\theta_1/\theta_2\in\QQ$. 
	We will prove this by contradiction, so let us assume 
	$\theta_1/\theta_2\in\RR\setminus\QQ$.
	
	We select $\RR$-linearly independent $x_1,y_1,x_2,y_2 \in\Vc\setminus\{0\}$ 
	with $D(x_j+\ie y_j)=\ie\theta_j(x_j+\ie y_j)
	\in\Vc_{\CC}=\CC\otimes_{\RR}\Vc$ for $j=1,2$, 
	and let us define $v:=x_1+y_1+x_2+y_2$. 
	We also denote $\Xc:=\spa_{\RR}\{x_1,x_2,y_1,y_2\}\subset \Vc$ and  define the $\RR$-linear isomorphism 
	$B\colon\Xc\to\CC^2$ with $B(x_1+y_1)=(1,0)$ and $B(x_2+y_2)=(0,1)$. 
	Then the point $w:=(1,1)\in\CC^2$ satisfies $w=Bv$ and, defining 
	$\gamma(t):=B\ee^{tD}B^{-1}$, one has 
	$$\gamma(t)=\begin{pmatrix}
	\ee^{\ie t\theta_1}  & 0 \\
	0                    & \ee^{\ie t\theta_2}
	\end{pmatrix}.$$
	Since $\theta_1/\theta_2\in\RR\setminus\QQ$, it is easily checked that the map $\gamma(\cdot)w$ is injective and on the other hand, 
	it is well known that its image is dense in the torus $\TT^2$. 
	 Therefore, if $\gamma(\RR)w$ were a locally closed subset of $\CC$, then $\gamma(\RR)w$ would be an open subset of the torus $\TT^2$, 
	which is not the case since $\gamma(\RR)$ is a subgroup of $\TT^2$ and $\gamma(\RR)\subsetneqq\TT^2$. 
	It then follows that neither $B^{-1}\gamma(\RR)w$ is a locally closed subset of $\Xc$, that is, the image of the map $\gamma_v$ fails to be a locally closed subset of $\Vc$, which is a contradiction with~\eqref{elliptic_item1}. 
	
	\eqref{elliptic_item2}$\Rightarrow$\eqref{elliptic_item1}
   It remains to show that $\Vc_0\subseteqq \Gamma$. 
 
Let $D_0=S_0+N_0$ be the Jordan decomposition with $S_0,N_0\in\End(\Vc_0)$, 
where $S_0$ is semisimple, $N_0$ is nilpotent, and $S_0N_0=N_0S_0$. 
 
 Specializing Setting~\ref{sett_decomp} for $G=(\RR,+)$ and $\pi(t):=\ee^{tS_0}$ for all $t\in \RR$, we may assume 
	that $\Vc_0=\CC^n$ and $S_0$ is a diagonal matrix having purely imaginary diagonal entries, say 
	$$S_0=\begin{pmatrix}
	\ie\theta_1 & & 0 \\
	&\ddots & \\
	0 & & \ie\theta_n
	\end{pmatrix}$$
	By the hypothesis \eqref{elliptic_item2}, there exist $\theta\in\RR\setminus\{0\}$ and $p_1,\dots,p_n\in\ZZ$, $q_1,\dots,q_n\in\NN$ with $\theta_j=\theta p_j/q_j$ for $j=1,\dots,n$. 
	Then,  
	the group morphism $\gamma\colon\RR\to U(n)$ ($\subseteq M_n(\CC)$), $\gamma(t):=\ee^{tS_0}$, satisfies $\gamma(2\pi q_1\cdots q_n/\theta)=\1$. 
	Therefore $\gamma(\RR)$ is a homeomorphic image of the compact group $\RR/\Ker\gamma\simeq\RR/\ZZ$, that is, $\gamma(\RR)$ is compact. 
	This directly implies that for every $v\in\Vc_0$ its orbit $\ee^{\RR S_0}v=\gamma(\RR)v$ is a compact subset of $\Vc_0$, and in particular is locally closed. 
 
When $D_0$ is semisimple,  $D_0= S_0$, and then it follows that $\Vc_0\subseteq \Gamma$. 
  
If  $D_0$ is not semisimple, then $N_0\ne 0$. 
For every $v\in\Vc_0\setminus\Ker N_0$,  $\lim\limits_{t\to\pm\infty}\Vert \ee^{tD_0}v\Vert=\infty$ 
(see for instance \cite[Step 2 in the proof of Lemma 2.13]{BB18}). 
This implies that$ \lim\limits_{t\to\pm\infty}\Vert \ee^{tD}v\Vert=\infty$ for every $v\in\Vc$ with $v_0\in\Vc_0\setminus\Ker N_0$,  hence the orbit $\ee^{\RR D}v$ is locally closed by Lemma~\ref{ACDO2.3_new}(\eqref{ACDO2.3_new_item5}$\Rightarrow$\eqref{ACDO2.3_new_item1}). 
On the other hand $D_0\vert_{\Ker {N_0}}= S_0\vert_{\Ker {N_0}} $, thus the orbit $\ee^{\RR D}v$ is locally closed even for $v\in\Ker N_0$, hence again $\Vc_0\subseteq\Gamma$. 

(b) Assume now that $S_D$ is not a closed subgroup in $\RR$. 

 If that $\Int\Gamma$ is not dense in $\Vc$, then 
Corollary~\ref{main_reg} or Remark~\ref{correm}, show that  if  $\Int\Gamma$ is not dense in $\Vc$ then $\Int\Gamma=\emptyset$. 
Since $\Vc\setminus \Vc_0\subseteq \Gamma$, it follows that $\Vc=\Vc_0$, or equivalently, $\sigma(D)\subset \ie \RR$. Then $D=D_0$.
	Assume that $D$ is not semisimple, that is, 
	$D=D_0= S_0+N_0$  with $N_0\ne 0$,  then the argument above shows that 
	$\Vc_0 \setminus \Ker N_0\subseteq \Gamma$, hence $\Int\Gamma\ne \emptyset$. 
	Therefore $D$ must be semisimple. 
	
	Assume now that  $D$ is semisimple and  $\sigma(D) \subset \ie \RR$. 
In this case, specializing Setting~\ref{sett_decomp} for $G=(\RR,+)$ and $\pi(t):=\ee^{tD}$ for all $t\in G$, we may assume 
that $\Vc=\CC^n$ and $D$ is a diagonal matrix having purely imaginary diagonal entries. 
Then, in the notation of Definition~\ref{sett_decomp_def},  $G_\nu(v)=G=\RR$ hence $\gg_\nu(v)=\RR$ for every $v\in\Vc$.
Then Theorem~\ref{ArCuOu16_Th3.6} along with the hypothesis that $S_D$ is not closed in $\RR$ imply that the orbit $\ee^{\RR D}v$ is not locally compact in $\Vc$ for every $v\in(\CC^\times)^n$. 
This shows that $\Int\Gamma=\emptyset$.	
\end{proof}

\begin{definition}\label{def_ax+b}
	\normalfont
	We define the Lie algebra $\gg_D:=\Vc\rtimes_D\RR$ with its Lie bracket 
	$$[(v_1,t_1),(v_2,t_2)]:=(t_1Dv_2-t_2Dv_1,0) $$
	and the Lie group $G_D:=\Vc\rtimes_D\RR$ with its product 
	given by 
	$$(v_1,t_1)\cdot(v_2,t_2)=(v_1+\ee^{t_1D}v_2, t_1+t_2).$$ 
Then $\gg_D$ is the Lie algebra of $G_D$.
\end{definition}

Assertion~\eqref{post-anti_item1} in the following theorem is well known and we give it a very short proof for the sake of completeness. 

\begin{theorem}\label{post-anti}
For arbitrary $D\in\End(\Vc)$, the following assertions holds: 
\begin{enumerate}[{\rm(i)}]
\item\label{post-anti_item1} 
	The set $S_D$ is closed in $\RR$ if and only if $C^*(G_D)$ is type~I. 
\item\label{post-anti_item2} 
   If $S_D$ is not closed in $\RR$, then   
 $D\in\End(\Vc)$ is semisimple and  $\sigma(D) \subset \ie \RR$  if and only if  $C^*(G_D)$ is antiliminary. 
\end{enumerate}
\end{theorem}

\begin{proof}
\eqref{post-anti_item1} 
Let $\Gamma^*$ be the set of all $\xi\in\Vc^*$ whose orbit $\ee^{\RR D^*}\xi$ is locally closed in $\Vc^*$, 
and recall that $\sigma(D^*)=\sigma(D)$. 
Then, by Lemma~\ref{elliptic}(a) applied for $D^*$ instead of $D$, the set $S_D$ is closed in $\RR$ if and only if $\Gamma^*=\Vc^*$. 
On the other hand, since $C^*(G_D)\simeq  \RR \ltimes_{\alpha_{D^*}} \Cc_0(\Vc^*)$ 
(see Remark~\ref{Will}), it follows either by \cite[Th. 3.3]{Go73} or by Lemma~\ref{anHuefWilliams} that $\Gamma^*=\Vc^*$ if and only if $C^*(G_D)$ is type~I.

\eqref{post-anti_item2}  
As above, by Lemma~\ref{elliptic}(b) applied for $D^*$ instead of $D$, 
if $S_D$ is not closed in $\RR$ then the conditions  $D\in\End(\Vc)$ is semisimple and $\sigma(D)\subset \ie \RR$ are satisfied if and only if $\Int\Gamma^*=\emptyset$. 
By Remark~\ref{correm}, this is the case if and only if $\Int \Gamma^*$ is not dense in $\Vc^*$, 
that is, by Corollary~\ref{main_reg}, if and only if $C^*(G_D)$ is antiliminary. 
\end{proof}

\begin{remark}
\label{III}
\normalfont 
For any $C^*$-algebra $\Ac$ let us denote by $F(\Ac)$ its set of factorial states, that is, the states $\varphi\in\Ac^*$ whose corresponding GNS representation $\pi_\varphi\colon \Ac\to\Bc(\Hc_\varphi)$ satisfies $\pi_\varphi(\Ac)'\cap\pi_\varphi(\Ac)''=\CC\1$. 
We also denote by $F_{\rm III}(\Ac)$ the set of all $\varphi\in F(\Ac)$ for which the factor~$\pi_\varphi(\Ac)''$ is type~III. 
We endow $F(\Ac)$ with its weak$^*$-topology inherited as a subset of the dual space~$\Ac^*$. 
Then 
the $C^*$-algebra $\Ac$ is antiliminary if and only if $F_{\rm III}(\Ac)$ is dense in $F(\Ac)$, by \cite[Th. 2.1]{ArBa86}. 
(The analogous result for type~II also holds if $\Ac$ is separable.)

Let $G$ be a locally compact group. 
It follows from the aforementioned result along with \cite[18.1.4]{Di69} that $C^*(G)$ is antiliminary if and only if every factor representation of $G$ is weakly contained in the type~III factor representations of $G$, 
in the sense that every coefficient of any factor representation of $G$ can be uniformly approximated on compacts by coefficients of type~III factor representations of $G$.

On the other hand, if $G$ is a connected, simply connected, solvable Lie group, and its quasi-dual $\overline{G}$ is endowed with the equivalence class of measures arising from the factor disintegrations of left regular representation as in \cite[8.4.3]{Di69}, then the subset $\overline{G}_{\rm III}$ of $\overline{G}$ corresponding to the type~III factor representations is negligible by \cite[Ch. IV, Cor. 7.2]{Pu71}. 

Nevertheless, $\overline{G}_{\rm III}$ may not be negligible from a topological point of view.  
Specifically, it follows by the above remarks that Theorem~\ref{post-anti}  provides specific examples illustrating both cases that can occur: 
either $\overline{G}$ is weakly contained in (hence weakly equivalent to) $\overline{G}_{\rm III}$ or not, 
depending on whether $C^*(G)$ is antiliminary or not.
\end{remark}

\begin{example}\label{Mautner}
\normalfont
Recalling the Mautner group $G_\theta:=G_{D_\theta}=\CC^2\rtimes_{\alpha_{D_\theta}}\RR$ defined by 
$$D_\theta=\begin{pmatrix}
\ie & 0 \\
0 & \ie\theta
\end{pmatrix}\in M_2(\CC) $$
for $\theta\in\RR\setminus\QQ$, it follows by Theorem~\ref{post-anti} that $C^*(G_\theta)$ is antiliminary. 
Also, $C^*(G_\theta)$ is not a simple $C^*$-algebra, in the sense that it has nontrivial closed 2-sided ideals. 
For instance, using the short exact sequence of groups
$$0\to \CC^2\hookrightarrow G_\theta\to \RR\to 0$$  
 one obtains the short exact sequence 
$$0\to\Jc\hookrightarrow C^*(G_\theta)\to \Cc_0(\RR)\to 0$$
for a suitable nontrivial ideal $\Jc$. 
The ideal $\Jc$ is antiliminary since $C^*(G_\theta)$ is antiliminary. 

On the other hand, every primitive ideal of $C^*(G_\theta)$ is maximal, 
 by \cite[Th.~2]{Pu73}. 
Equivalently, for every irreducible $*$-representation $\pi\colon C^*(G_\theta)\to\Bc(\Hc_\pi)$ its corresponding primitive quotient $C^*(G_\theta)/\Ker\pi\simeq \pi(C^*(G_\theta))$ is a simple $C^*$-algebra. 
All the simple $C^*$-algebras that arise in this way are quasidiagonal,  hence
$C^*(G_\theta)$ is strongly quasidiagonal (see \cite{BB20}).
It was already known that the $C^*$-algebra  
$C^*(G_\theta)$ is quasidiagonal, 
as noted in \cite[end of Sect. 2]{BB18}. 
\end{example}

\subsection*{Acknowledgment} 
We wish to thank the Referees for their generous remarks and suggestions, 
which in particular led to Proposition~\ref{nonabelian} and Example~\ref{Dixmier}.


\begin{thebibliography}{100000000}
	
\bibitem[ACM13]{ACM13}
\textsc{E.~Andruchow, G.~Corach, M.~Mbekhta}, 
 A geometry for the set of split operators. 
 \textit{Integral Equations Operator Theory} \textbf{77} (2013), no. 4, 559--579.
	
\bibitem[aHWi02]{aHWi02}
{\sc A.~an Huef, D.P.~Williams}, 
Ideals in transformation-group $C^\ast$-algebras. 
{\it J. Operator Theory} {\bf 48} (2002), no. 3, suppl., 535--548. 

\bibitem[ArBa86]{ArBa86}
{\sc R.J.~Archbold, C.J.K.~Batty}, 
On factorial states of operator algebras. III. 
{\it J. Operator Theory} {\bf 15} (1986), no. 1, 53--81. 

\bibitem[ArKS12]{ArKS12}
{\sc R.J.~Archbold, E.~Kaniuth, D.W.B.~Somerset}, 
Norms of inner derivations for multiplier algebras of $C^*$-algebras and group $C^*$-algebras. 
{\it J. Funct. Anal.} {\bf 262} (2012), no. 5, 2050--2073.

\bibitem[ArKS15]{ArKS15}
{\sc R.J.~Archbold, E.~Kaniuth, D.W.B.~Somerset}, 
Norms of inner derivations for multiplier algebras of $C^*$-algebras and group $C^*$-algebras, II. 
{\it Adv. Math.} {\bf 280} (2015), 225--255. 

\bibitem[AC20]{AC20}
	\textsc{D.~Arnal, B.~Currey III}, 
	{\it Representations of solvable Lie groups}. 
	Cambridge University Press, 2020.


\bibitem[ACD12]{ArCuDa12}
{\sc D.~Arnal, B.~Currey, B.~Dali}, 
Canonical coordinates for a class of solvable groups. 
{\it Monatsh. Math.} {\bf 166}  (2012), no. 1, 19--55. 
	
\bibitem[ACDO13]{ArCuDaOu13}
{\sc D.~Arnal, B.~Currey, B.~Dali, V.~Oussa}, 
Regularity of abelian linear actions. 
In: A.~Mayeli, A.~Iosevich, P.E.T.~Jorgensen, G.~\'Olafsson (eds.),   {\it Commutative and noncommutative harmonic analysis and applications}, 
Contemp. Math., 603, Amer. Math. Soc., Providence, RI, 2013, pp.~89--109. 

\bibitem[ACO16]{ArCuOu16}
{\sc D.~Arnal, B.~Currey, V.~Oussa}, 
Characterization of regularity for a connected Abelian action. 
{\it Monatsh. Math.} {\bf 180} (2016),  no.~1, 1--37.

\bibitem[ACD19]{ArCuDa19}
\textsc{D.~Arnal, B.~Currey, B.~Dali},
The Plancherel formula for an inhomogeneous vector group. 
\textit{J. Fourier Anal. Appl.} \textbf{25} (2019), no. 6, 2837--2876. 


\bibitem[AuKo71]{AuKo71}
{\sc L.~Auslander, B.~Kostant}, 
Polarization and unitary representations of solvable Lie groups. 
{\it Invent. Math.} {\bf 14} (1971), 255--354.

\bibitem[Ba98]{Ba98}
{\sc P.~Baguis}, 
Semidirect products and the Pukanszky condition. 
{\it J. Geom. Phys.} {\bf 25} (1998), no. 3--4, 245--270.

\bibitem[BaRa86]{BaRa86}
{\sc A.O.~Barut, R.~R\c aczka}, 
{\it Theory of group representations and applications}. 
Second edition. World Scientific Publishing Co., Singapore, 1986. 

\bibitem[BB17]{BB17}
{\sc I.~Belti\c t\u a, D.~Belti\c t\u a}, 
Nonlinear oblique projections. 
{\it Linear Algebra Appl.} {\bf 533} (2017), 451--467.

\bibitem[BB18]{BB18}
{\sc I.~Belti\c t\u a, D.~Belti\c t\u a}, 
Quasidiagonality of $C^*$-algebras of solvable Lie groups. 
{\it Integral Equations Operator Theory} {\bf 90} (2018),  no.~1, Art.~5, 21 pp. 

\bibitem[BB20]{BB20}
{\sc I.~Belti\c t\u a, D.~Belti\c t\u a}, 
AF-embeddability for Lie groups with $T_1$ primitive ideal spaces.  
{\it J.~London Math. Soc} (to appear; see arXiv:2004.11010v3 [math.OA]).


\bibitem[Bo74]{Bo74}
{\sc N.~Bourbaki}, 
{\it Topologie g\'en\'erale}. Chap. 5 \`a 10. Nouvelle \'edition. Hermann, Paris, 1974.

\bibitem[BCFM15]{BrCFM15}
{\sc J.~Bruna, J.~Cuf\'\i, H.~F\"uhr, M.~Mir\'o}, 
Characterizing abelian admissible groups. 
J. Geom. Anal. {\bf 25} (2015), no. 2, 1045--1074. 

\bibitem[Ca05]{Ca05}
{\sc R.W.~Carter}, 
\textit{Lie algebras of finite and affine type}. 
Cambridge Studies in Advanced Mathematics, 96. Cambridge University Press, Cambridge, 2005. 


\bibitem[CK14]{CW14}
\textsc{F.~Colonius, W.~Kliemann}, 
\textit{Dynamical systems and linear algebra}. 
Graduate Studies in Mathematics \textbf{158}. American Mathematical Society, Providence, RI, 2014

\bibitem[CG90]{CG90}
\textsc{L.J.~Corwin, F.P.~Greenleaf}, 
\textit{Representations of nilpotent Lie groups and their applications}. 
Cambridge Studies in Advanced Mathematics, 18. Cambridge University Press, Cambridge, 1990.

\bibitem[Di61]{Di61}
	\textsc{J.~Dixmier}, Sur le rev\^etement universel d'un groupe de Lie de type I. {\it C. R. Acad. Sci. Paris} \textbf{252} (1961), 2805--2806.  



\bibitem[Di69]{Di69}
{\sc J.~Dixmier}, 
{\it Les $C^{\ast}$-alg\`ebres et leurs repr\'esentations}. 
Deuxi\`eme \'edition. Cahiers Scientifiques, Fasc. XXIX. Gauthier-Villars \'Editeur, Paris, 1969




\bibitem[DGL05]{DGL05}
\textsc{G.M.~Diaz-Toca, L.~Gonzalez-Vega, H.~Lombardi}, 
Generalizing Cramer's rule: solving uniformly linear systems of equations. 
{\it SIAM J. Matrix Anal. Appl.} {\bf 27} (2005), no. 3, 621--637. 

\bibitem[GvL96]{GvL96}
{\sc G.H.~Golub, C.F.~van Loan}, 
{\it Matrix computations}. Third edition. Johns Hopkins Studies in the Mathematical Sciences. 
Johns Hopkins University Press, Baltimore, MD, 1996. 


\bibitem[Go73]{Go73}
{\sc E.C.~Gootman}, 
The type of some $C^*$- and $W^*$-algebras associated with transformation groups. 
{\it Pacific J. Math.} {\bf 48} (1973), 93--106.

\bibitem[Ho65]{Ho65}
{\sc G.~Hochschild}, 
{\it The structure of Lie groups}. 
Holden-Day, Inc., San Francisco-London-Amsterdam, 1965.





\bibitem[Pe94]{Pe94}
{\sc N.V.~Pedersen}, 
Orbits and primitive ideals of solvable Lie algebras. 
{\it Math. Ann.} {\bf 298} (1994),  no.~2, 275--326.

\bibitem[Pu71]{Pu71}
{\sc L.~Pukanszky}, 
Unitary representations of solvable Lie groups. 
{\it Ann. Sci. \'Ecole Norm. Sup. (4)} {\bf 4} (1971), 457--608. 

\bibitem[Pu73]{Pu73}
{\sc L.~Pukanszky}, 
The primitive ideal space of solvable Lie groups. 
{\it Invent. Math.} {\bf 22} (1973), 75--118. 

\bibitem[Ra75]{Ra75}
{\sc J.H.~Rawnsley}, 
Representations of a semi-direct product by quantization. 
{\it Math. Proc. Cambridge Philos. Soc.} {\bf 78} (1975), no. 2, 345--350.

\bibitem[Wi07]{Wi07}
{\sc D.P.~Williams}, 
{\it Crossed products of $C^*$-algebras}. 
Mathematical Surveys and Monographs, 134. American Mathematical Society, Providence, RI, 2007.

\end{thebibliography}
\end{document}